\documentclass{amsart}
\chardef\bslash=`\\

\usepackage{amssymb,amsmath,amsfonts,amsthm,epsfig,amscd,stmaryrd,enumitem}
\usepackage{stmaryrd}
\usepackage[all,cmtip,poly]{xy}
\usepackage{xcolor}
\usepackage{tikz}
\usepackage{hyperref}

\usepackage{silence}
\usepackage{microtype}

\newtheorem{theorem}[subsection]{Theorem}

\newtheorem{lemma}[subsection]{Lemma}

\newtheorem{introthm}{Theorem}

\newtheorem{introcor}[introthm]{Corollary}

\newtheorem{proposition}[subsection]{Proposition}

\newtheorem{defn}[subsection]{Definition}
\theoremstyle{remark}

\numberwithin{equation}{subsection}

\newif\iffinalrun
\iffinalrun
\else
 \fi

\iffinalrun
  \newcommand{\need}[1]{}
  \newcommand{\mar}[1]{}
\else
  \newcommand{\need}[1]{{\tiny *** #1}}
  \newcommand{\mar}[1]{\marginpar{\raggedright\tiny  #1}}\fi

\renewcommand\mathbb{\mathbf}

\newcommand{\rec}{{\operatorname{rec}}}

\newcommand{\rbar}{\overline{r}}

\newcommand{\wv}{{\widetilde{v}}}

\renewcommand{\ell}{l}

\def\PSL{\mathrm{PSL}}
\def\PGL{\mathrm{PGL}}

\def\Iw{\mathrm{Iw}}

\newcommand{\ad}{\operatorname{ad}}
\newcommand{\diag}{\operatorname{diag}}
\newcommand{\tr}{\operatorname{tr}}
\newcommand{\Sp}{\operatorname{Sp}}

\newcommand{\A}{\mathbf{A}}
\newcommand{\bA}{\ensuremath{\mathbf{A}}}

\newcommand{\bC}{\ensuremath{\mathbf{C}}}

\newcommand{\bF}{\ensuremath{\mathbf{F}}}
\newcommand{\bG}{\ensuremath{\mathbf{G}}}

\newcommand{\bQ}{\ensuremath{\mathbf{Q}}}
\newcommand{\Q}{\QQ}
\newcommand{\QQ}{{\mathbb Q}}

\newcommand{\bT}{\ensuremath{\mathbf{T}}}

\newcommand{\bZ}{\ensuremath{\mathbf{Z}}}

\newcommand{\bbZ}{\ensuremath{\mathbf{Z}}}
\newcommand{\bbQ}{\ensuremath{\mathbf{Q}}}

\newcommand{\cC}{{\mathcal C}}

\newcommand{\cH}{{\mathcal H}}

\newcommand{\cO}{{\mathcal O}}

\newcommand{\cT}{{\mathcal T}}

\newcommand{\ffrm}{{\mathfrak m}}

\newcommand{\frb}{{\mathfrak b}}

\newcommand{\frakp}{\mathfrak{p}}
\newcommand{\p}{\frakp}

\newcommand{\fra}{{\mathfrak a}}

\newcommand{\C}{\mathbf{C}}

\newcommand{\Qp}{\Q_p}

\DeclareMathOperator{\Lift}{Lift}
\DeclareMathOperator{\Def}{Def}

\DeclareMathOperator{\Gal}{Gal}
\newcommand{\GL}{\mathrm{GL}}

\DeclareMathOperator{\Hom}{Hom}

\DeclareMathOperator{\Ind}{Ind}

\DeclareMathOperator{\SL}{SL}
\DeclareMathOperator{\Spec}{Spec}

\DeclareMathOperator{\Supp}{Supp}

\DeclareMathOperator{\Sym}{Sym}

\newcommand{\Frob}{\mathrm{Frob}}

\newcommand{\Art}{{\operatorname{Art}}}
\newcommand{\Res}{\operatorname{Res}}

\newcommand{\doubleslash}{/\kern-0.2em{/}}

\begin{document}

\author{James Newton and Jack A. Thorne}
\title[Symmetric power functoriality, II]{Symmetric power functoriality for holomorphic modular forms, II}
\begin{abstract} Let $f$ be a cuspidal Hecke eigenform without complex multiplication. We prove the automorphy of the symmetric power lifting $\Sym^n f$ for every $n \geq 1$.
\end{abstract}
\maketitle
\setcounter{tocdepth}{1}
\tableofcontents
\section{Introduction}
Let $F$ be a number field, and let $\pi$ be a cuspidal automorphic representation of $\GL_2(\bA_F)$. Langlands's functoriality principle predicts the existence, for any $n \geq 1$, of an automorphic representation $\Sym^n \pi$ of $\GL_{n+1}(\bA_F)$, characterized by the requirement that for any place $v$ of $F$, the Langlands parameter of $(\Sym^n \pi)_v$ is the image  of the Langlands parameter of $\pi_v$ under the $n$th symmetric power $\Sym^n : \GL_2 \to \GL_{n+1}$ of the standard  representation of $\GL_2$.  

For a more detailed  discussion of the context surrounding this problem, including known results by other authors, we refer the reader to the introduction of \cite{New19b},  of  which this paper  is a continuation.  In that paper we studied the problem of symmetric power functoriality in the case that $F = \bQ$ and $\pi$ is regular algebraic (in which case $\pi$ corresponds to a twist of a cuspidal Hecke eigenform $f$ of weight $k \geq 2$, cf.~\cite[\S 3]{Gel75}). We established the existence of the symmetric power liftings $\Sym^n \pi$ under the assumption that there is no prime $p$ such that $\pi_p$ is supercuspidal. (This includes the case that $f$ has level $\SL_2(\bZ)$.) 

In  this paper we remove this assumption, proving the following theorem:
\begin{introthm}\label{introthm_sympowers}
	Let $\pi$ be a regular algebraic, cuspidal automorphic representation of $\GL_2(\bA_\bQ)$. Suppose that $\pi$ is non-CM. Then for each integer $n \geq 1$, $\Sym^{n} \pi$ exists, as a regular algebraic, cuspidal automorphic representation of $\GL_{n+1}(\bA_\bQ)$.
\end{introthm}
In the `missing' cases of $\pi$ which are holomorphic limit of discrete series at $\infty$ or CM, the existence of $\Sym^n \pi$ for all $n$ is well known, although of course $\Sym^n\pi$ is usually not cuspidal. The most difficult case of icosahedral weight one eigenforms (\cite[Theorem 6.4]{Kim2004}) requires Kim and Shahidi's results on tensor product and symmetric power functoriality. We provide some details in Appendix \ref{wt1}.

Using the modularity of elliptic curves over $\bQ$ \cite{MR1839918}, we deduce the following corollary:
\begin{introcor}
	Let $E$ be an elliptic curve over $\bQ$ without complex multiplication. Then, for each integer $n \geq 2$, the completed symmetric power $L$-function $\Lambda(\Sym^n E, s)$ as defined in e.g.\ \cite{Dum09}, admits an analytic continuation to the entire complex plane.
\end{introcor}
Our strategy to prove Theorem \ref{introthm_sympowers} is inspired by the proof of Serre's conjecture \cite{Kha09}. There one takes as given Serre's conjecture in the level 1 case (i.e. for every prime  number $p$,  the residual modularity of odd irreducible representations 
\[ \overline{\rho}  : \Gal(\overline{\bQ}/\bQ) \to \GL_2(\overline{\bF}_p) \]
 unramified outside $p$), proved in \cite{Kha06}, and hopes to reduce the general case to this one by induction on the number of primes  away from $p$ at which $\overline{\rho}$ is ramified.

Here we associate to any regular algebraic, cuspidal automorphic representation $\pi$ of $\GL_2(\bA_\bQ)$ the set $sc(\pi)$ of primes $p$  such that $\pi_p$ is supercuspidal. Fixing $n \geq 1$, we prove the existence  of $\Sym^n \pi$ by induction on the cardinality of the set $sc(\pi)$, the case $|sc(\pi)| = 0$ being exactly the main result of \cite{New19b}.

Our induction argument uses congruences between automorphic representations. If $p$ is a prime and $\iota : \overline{\bQ}_p \to \bC$ is an isomorphism, then there is an associated Galois representation 
\[ r_{\pi, \iota}: \Gal(\overline{\bQ}/\bQ) \to \GL_2(\overline{\bQ}_p) \]
 and its mod $p$ reduction 
 \[ \overline{r}_{\pi, \iota} : \Gal(\overline{\bQ}/\bQ) \to \GL_2(\overline{\bF}_p). \]
  If $\pi'$ is another regular algebraic, cuspidal automorphic representation of $\GL_2(\bA_\bQ)$, and there is an isomorphism 
\[ \overline{r}_{\pi, \iota} \cong \overline{r}_{\pi', \iota}\]
(in other words, a congruence modulo $p$ between $\pi$ and $\pi'$), then passage to symmetric powers gives an isomorphism 
\[ \Sym^n \overline{r}_{\pi, \iota} \cong \Sym^n \overline{r}_{\pi', \iota}. \]
If $\Sym^n \pi'$ is known to exist, and the image of the representation  $\Sym^n \overline{r}_{\pi',  \iota}$ is sufficiently non-degenerate (for example, irreducible), then automorphy lifting theorems (such as those proved in \cite{BLGGT}) can be used to deduce the automorphy of $\Sym^n r_{\pi, \iota}$, hence the existence of $\Sym^n \pi$. 

If $p$ is a prime such that $\pi_p$ is supercuspidal, it may be possible to  choose $\pi'$ so that $sc(\pi') = sc(\pi) - \{ p  \}$, opening the way to an  induction argument. This idea of `killing ramification' plays a significant role in \cite{Kha09}.

The difficulty in applying this approach here is that if $p \leq n$ then the representation $\Sym^n \overline{r}_{\pi, \iota}$ is never irreducible, so the automorphy lifting theorems proved in \cite{BLGGT} do not apply. (The automorphy lifting theorems for residually reducible representations proved in \cite{All19} apply only for ordinary representations, a possibility which is ruled out if $\pi_p$ is supercuspidal.) This approach might perhaps yield the existence of $\Sym^n \pi$ when $p > n$ for every $p \in sc(\pi)$, but to get a result like Theorem \ref{introthm_sympowers} a new idea is required.

Here we prove a new kind of automorphy lifting theorem, Theorem \ref{thm_automorphy_of_symmetric_power}, specially tailored to the problem of symmetric powers (although we hope that these ideas will also be useful for other cases of Langlands functoriality). We consider the morphism $P \to R$, where $R$ is the  universal deformation ring of the (supposed irreducible) representation $\overline{r}_{\pi, \iota}$ and $P$ is the universal pseudodeformation ring of the pseudocharacter associated to the symmetric power $\Sym^n \overline{r}_{\pi, \iota}$; the morphism $P \to R$ is the universal one classifying the pseudocharacter of $\Sym^n$ of the universal deformation of $\overline{r}_{\pi, \iota}$. A version of the Taylor--Wiles--Kisin patching argument upgrades this to a commutative diagram
\[ \xymatrix{ P_\infty \ar[d] \ar[r] &  R_\infty \ar[d] \\ P \ar[r] & R, } \]
where $P_\infty, R_\infty$ are `patched deformation rings' and the vertical arrows are surjections. Since $\overline{r}_{\pi, \iota}$ is assumed to be irreducible, the arguments of \cite{Kis09, MR2551765} show that $R_\infty$ is a domain which acts faithfully on a space of patched (rank-2) modular forms. 

The essential additional ingredient is the main result of \cite{New19a}, which shows that $\Spec P$ is regular (of dimension 0) at the point corresponding to the pseudocharacter of the representation $\Sym^n r_{\pi', \iota}$; this in turn implies that $\Spec P_\infty$ is regular at the image of this point in $\Spec P_\infty$, and allows us to deduce that the image of $\Spec R_\infty \to \Spec P_\infty$ is contained in the support of a space of patched (rank-$(n+1)$) modular forms, leading to a proof of Theorem \ref{thm_automorphy_of_symmetric_power}. Our a priori  knowledge about the ring $P$ obviates the need to kill the dual Selmer group of $\Sym^n \overline{r}_{\pi, \iota}$.

To actually prove Theorem \ref{introthm_sympowers}, we combine Theorem \ref{thm_automorphy_of_symmetric_power} with a modified version of the `killing ramification' technique of \cite{Kha09}, based on a variation of the notion of `good dihedral' representation introduced in that paper. This is not quite routine since we need our `good dihedral' automorphic representations $\pi$ to have the property that, if $q$ is the good dihedral prime, then there is an isomorphism $\iota_q : \overline{\bQ}_q \to \bC$ such that $\overline{r}_{\pi, \iota_q}$ has large image. We achieve this by introducing Steinberg type ramification at another auxiliary prime $r$, which is acceptable since the presence of $r$ does not affect the set $sc(\pi)$. We call an automorphic representation $\pi$ that comes equipped with the requisite auxiliary primes `seasoned' (see Definition \ref{def_seasoned}). 

\textbf{Acknowledgements.}  J.T.'s work received funding from the European 
Research Council (ERC) under the European Union's Horizon 2020 research and 
innovation programme (grant agreement No 714405). We thank Toby Gee and an anonymous referee for comments on earlier versions of this manuscript.

\textbf{Notation.}
If $F$ is a perfect field, we generally fix an algebraic closure $\overline{F} / F$ and write $G_F$ for the absolute Galois group of $F$ with respect to this choice. When the characteristic of $F$ is not equal to $p$, we write $\epsilon : G_F \to \bbZ_p^\times$ for the $p$-adic cyclotomic character. We write $\zeta_n \in \overline{F}$ for a fixed choice of primitive $n^\text{th}$ root of unity (when this exists). If $F$ is a number field, then we will also fix embeddings $\overline{F} \to \overline{F}_v$ extending the map $F\to F_v$ for each place $v$ of $F$; this choice determines a homomorphism $G_{F_v} \to G_F$. When $v$ is a finite place, we will write $W_{F_v} \subset G_{F_v}$ for the Weil group, $\cO_{F_v} \subset F_v$ for the valuation ring, $\varpi_v \in \cO_{F_v}$ for a fixed choice of uniformizer, $\Frob_v \in G_{F_v}$ for a fixed choice of (geometric) Frobenius lift, $k(v) = \cO_{F_v} / (\varpi_v)$ for the residue field, and $q_v = \# k(v)$ for the cardinality of the residue field. If $R$ is a ring and  $\alpha \in R^\times$, then we write $\operatorname{ur}_\alpha : W_{F_v} \to R^\times$ for the unramified character which sends $\Frob_v$ to $\alpha$. When $v$ is a real place, we write $c_v \in G_{F_v}$ for complex conjugation. If $S$ is a finite set of finite places of $F$ then we write $F_S / F$ for the maximal subextension of $\overline{F}$ unramified outside $S$ and $G_{F, S} = \Gal(F_S / F)$. 

If $p$ is a prime, then we call a coefficient field a finite extension $E / \bbQ_p$ contained inside our fixed algebraic closure $\overline{\bbQ}_p$, and write $\cO$ for the valuation ring of $E$, $\varpi \in \cO$ for a fixed choice of uniformizer, and $k = \cO / (\varpi)$ for the residue field. We write $\cC_\cO$ for the category of complete Noetherian local $\cO$-algebras with residue field $k$. If $G$ is a profinite group and $\rho : G \to \GL_n(\overline{\bQ}_p)$ is a continuous representation, then we write $\overline{\rho} : G \to \GL_n(\overline{\bF}_p)$ for the associated semisimple residual representation (which is well-defined up to conjugacy). If $F$ is a number field, $v$ is a finite place of $F$, and $\rho, \rho' : G_{F_v} \to \GL_n(\cO_{\overline{\bQ}_p})$ are continuous representations, which are potentially crystalline if $v | p$, then we use the notation $\rho \sim \rho'$ established in \cite[\S 1]{BLGGT} (which indicates that these two representations define points on a common component of a suitable deformation ring).

We write $T_n \subset B_n \subset \GL_n$ for the standard diagonal maximal torus and upper-triangular Borel subgroup. Let $K$ be a non-archimedean characteristic $0$ local field, 
and let $\Omega$ 
be an algebraically 
closed field of characteristic 0. If $\rho : G_K \to \GL_n(\overline{\bbQ}_p)$ is a continuous 
representation (which is de Rham if $p$ equals the residue characteristic of 
$K$), then we write $\mathrm{WD}(\rho) = (r, N)$ for the associated 
Weil--Deligne representation of $\rho$, and $\mathrm{WD}(\rho)^{F-ss}$ for its 
Frobenius semisimplification. We use the cohomological normalisation of 
class field theory: it is the isomorphism $\Art_K: K^\times \to W_K^{ab}$ which 
sends uniformizers to geometric Frobenius elements. When $\Omega = \bC$, we 
have the local Langlands correspondence $\rec_{K}$ for $\GL_n(K)$: a bijection 
between the sets of isomorphism classes of irreducible, admissible 
$\bC[\GL_n(K)]$-modules and Frobenius-semisimple Weil--Deligne representations 
over $\bC$ of rank $n$. In general, we have the Tate normalisation of the local 
Langlands correspondence 
for $\GL_n$ as described in \cite[\S 
2.1]{Clo14}. When $\Omega = 
\bC$, we have $\rec^T_K(\pi) = \rec_K(\pi \otimes | \cdot |^{(1-n)/2})$. 

If $F$ is a number field and $\chi : F^\times \backslash \A_F^\times \to \bC^\times$ is a Hecke character of type $A_0$ (equivalently: algebraic), then for any isomorphism $\iota : \overline{\bQ}_p \to \bC$ there is a continuous character $r_{\chi, \iota} : G_F \to \overline{\bQ}_p^\times$ which is de Rham at the places $v | p$ of $F$ and such that for each finite place $v$ of $F$, $\mathrm{WD}(r_{\chi, \iota}) \circ \Art_{F_v} = \iota^{-1} \chi|_{F_v^\times}$. Conversely, if $\chi' : G_F \to \overline{\bQ}_p^\times$ is a continuous character which is de Rham and unramified at all but finitely many places, then there exists a Hecke character $\chi : F^\times \backslash \A_F^\times \to \bC^\times$ of type $A_0$ such that $r_{\chi, \iota} = \chi'$. In this situation we abuse notation slightly by writing $\chi = \iota \chi'$.

If $F$ is a CM or totally real number field and $\pi$ is an automorphic representation of $\GL_n(\A_F)$, we say that $\pi$ is regular algebraic if $\pi_\infty$ has the same infinitesimal character as an irreducible algebraic representation $W$ of $(\Res_{F/ \bQ} \GL_n)_\bC$. 

If $\pi$ is cuspidal, regular algebraic, and polarizable, in the sense of \cite{BLGGT}, then for any isomorphism $\iota : \overline{\bQ}_p \to \bC$ there exists a continuous, semisimple representation $r_{\pi, \iota} : G_F \to \GL_n(\overline{\bQ}_p)$ such that for each finite place $v$ of $F$, $\mathrm{WD}(r_{\pi, \iota}|_{G_{F_v}})^{F-ss} \cong \rec_{F_v}^T(\iota^{-1} \pi_v)$ (see e.g.\ \cite{Caraianilp}). (When $n = 1$, this is compatible with our existing notation.) We use the convention that the Hodge--Tate weight of the cyclotomic character is $-1$. 

We use special terminology in the case $n = 2$: if $k \geq 2$ is an integer, we say that $\pi$  has weight $k$ if we can take $W = (\otimes_{\tau \in \Hom(F, \bC)} \Sym^{k-2} \bC^2)^\vee$. (If $F$ is totally real, then the cuspidal automorphic representations of weight $k$ are those which are associated to cuspidal Hilbert modular forms of parallel weight $k$.) In this case the Hodge--Tate weights of $r_{\pi, \iota}$ with respect to any embedding $\tau :F \to \overline{\bQ}_p$ are $\{0, k-1\}$ and the character $\epsilon^{k-1} \det r_{\pi, \iota}$ has finite order.

If $F$ is a number field, $G$ is a reductive group over $F$, $v$ is a finite place of $F$, and $U_v$ is an open compact subgroup of $G(F_v)$, then we write $\cH(G(F_v), U_v)$ for the convolution algebra of compactly supported $U_v$-biinvariant functions $f : G(F_v) \to \bZ$ (convolution defined with respect to the Haar measure on $G(F_v)$ which gives $U_v$ volume 1). Then $\cH(G(F_v), U_v)$ is a free $\bZ$-module, with basis given by the characteristic functions $[U_v g_v U_v]$ of double cosets for $g_v \in U_v \backslash G(F_v) / U_v$.

If $1 \leq i \leq n$, let $\alpha_{\varpi_v, i} = \diag(\varpi_v, \dots, \varpi_v, 1, \dots, 1) \in \GL_n(F_v)$ (where there are $i$ occurrences of $\varpi_v$ on the diagonal). We define 
\[ T_v^{(i)} = [\GL_n(\cO_{F_v}) \alpha_{\varpi_v, i} \GL_n(\cO_{F_v})] \in \cH(\GL_n(F_v), \GL_n(\cO_{F_v})). \]
We write $\Iw_v \subset \GL_n(\cO_{F_v})$ for the standard Iwahori subgroup (elements which are upper-triangular modulo $\varpi_v$) and $\Iw_{v, 1} \subset \Iw_v$ for the kernel of the natural map $\Iw_v \to (k(v)^\times)^n$ given by reduction modulo $\varpi_v$, then projection to the diagonal. If $U_v \subset \Iw_v$ is a subgroup containing $\Iw_{v, 1}$, and $1 \leq i \leq n$, then we define
\[ U_{\varpi_v}^{(i)} = [U_v \alpha_{\varpi_v, i} U_v] \in \cH(\GL_n(F_v), U_v). \]

\section{An automorphy lifting theorem for symmetric power representations}\label{sec_ALT_for_sym_powers} 

Let $p$ be a prime and let $F$ be a totally real field. Fix an isomorphism $\iota : \overline{\bQ}_p \to \bC$. Let $n \geq 1$. Let $\pi$ be a regular algebraic, cuspidal automorphic representation of $\GL_2(\bA_F)$ satisfying the following conditions:
\begin{itemize}
\item $\pi$ has weight 2 and is non-CM.
\item For each place $v | p$ of $F$, $r_{\pi, \iota}|_{G_{F_v}}$ is not ordinary, in the sense of \cite[ \S 5.1]{Tho16}. Note that, together with the assumption that $\pi$ has weight 2, this implies that $r_{\pi, \iota}|_{G_{F_v}}$ is potentially Barsotti--Tate.
\item Let $\operatorname{Proj} \overline{r}_{\pi, \iota} : G_F \to \PGL_2(\overline{\bF}_p)$ denote the projective representation associated to $\overline{r}_{\pi, \iota}$. Then there exists $a \geq 1$ such that $p^a > \max(5, 2n-1)$ and there is a sandwich
\[ \PSL_2(\bF_{p^a}) \subset \operatorname{Proj} \overline{r}_{\pi, \iota}(G_F) \subset \PGL_2(\bF_{p^a}), \]
up to conjugacy in $\PGL_2(\overline{\bF}_p)$.
\end{itemize}
We impose the final condition to ensure that we can choose Taylor--Wiles primes such that the image of the corresponding Frobenius element under $\Sym^{n-1}\rbar_{\pi,\iota}$ is regular semisimple. 

The aim of this section is prove the following theorem:
\begin{theorem}\label{thm_automorphy_of_symmetric_power}
Suppose that there exists another regular algebraic, cuspidal automorphic representation $\pi'$ of $\GL_2(\bA_F)$ such that the following conditions are satisfied:
\begin{enumerate}
\item $\pi'$ has weight 2 and is non-CM.
\item For each place $v | p$ of $F$, $r_{\pi', \iota}|_{G_{F_v}}$ is not ordinary.
\item There is an isomorphism $\overline{r}_{\pi', \iota} \cong \overline{r}_{\pi, \iota}$.
\item For each place $v \nmid p$ of $F$, $\pi_v$ is a character twist of the Steinberg representation if and only if $\pi'_v$ is.
\item $\Sym^{n-1} r_{\pi', \iota}$ is automorphic.
\end{enumerate}
Then $\Sym^{n-1} r_{\pi, \iota}$ is automorphic. 
\end{theorem}
We begin with a preliminary reduction. Let $E / \bQ_p$ be a coefficient field. After possibly enlarging $E$, we can find conjugates $r, r'$ of $r_{\pi, \iota}, r_{\pi', \iota}$ respectively which take values in $\GL_2(\cO)$. We can also assume that the eigenvalues of each element in the image of $\overline{r}$ lie in $k$. After passage to a soluble totally real extension, we can assume that the following additional conditions are satisfied:
\begin{enumerate}
\setcounter{enumi}{5}
\item $[F : \bQ]$ is even.
\item\label{cond:det} $\det r_{\pi, \iota} = \det r_{\pi', \iota} = \epsilon^{-1}$.
\item For each place $v | p$ of $F$, $\pi_v$ and $\pi'_v$ are unramified.
\item For each finite place $v \nmid p$ of $F$, $\pi_v$ and $\pi'_v$ are Iwahori-spherical. The number of places such that $\pi_v$ is ramified is even.
\item Let $S_p$ denote the set of $p$-adic places of $F$ and $\Sigma$ the set of places $v$ such that $\pi_v$ is ramified.  Let $S = S_p \cup \Sigma$. Then for each $v \in S$, $\overline{r}|_{G_{F_v}}$ is trivial. For each $v \in \Sigma$, $q_v \equiv 1 \text{ mod }p$ and $\pi_v, \pi'_v$ are isomorphic to the Steinberg representation (not just up to twist --- note that condition (\ref{cond:det}) already implies that any such twist is by a quadratic character).
\item There exists an everywhere unramified CM quadratic extension $K / F$, with each place $v \in S$  split in $K$.
\end{enumerate}
Let $\Pi = \pi_K$. Then $\Pi$ is RACSDC (i.e.\ regular algebraic, conjugate self-dual, and cuspidal). We will show that the representation $\Sym^{n-1} r_{\Pi, \iota}$ is automorphic; this will imply Theorem \ref{thm_automorphy_of_symmetric_power}, by soluble descent. We let $\Pi' = \pi'_K$. Then $\Pi'$ is also RACSDC, there is an isomorphism $\overline{r}_{\Pi, \iota} \cong \overline{r}_{\Pi', \iota}$, and $\Sym^{n-1} r_{\Pi', \iota}$ is automorphic. We write $\Pi'_n$ for  the RACSDC automorphic representation of $\GL_n(\bA_{K})$ such that $r_{\Pi'_n, \iota} \cong \Sym^{n-1}  r_{\Pi', \iota}$.

Recall that $\cC_\cO$ denotes the category of complete Noetherian local 
$\cO$-algebras with residue field $k$. If $v$ is a place of $F$, we write $R_v^\square \in \cC_\cO$ for 
the object representing the functor $\Lift_v : \cC_\cO \to \operatorname{Sets}$ 
which associates to $A \in \cC_\cO$ the set of homomorphisms $\widetilde{r} : 
G_{F_v} \to \GL_2(A)$ lifting $\overline{r}|_{G_{F_v}}$ (i.e.\ such that 
$\widetilde{r} \text{ mod }\ffrm_A = \overline{r}$) such  that $\det 
\widetilde{r} = \epsilon^{-1}|_{G_{F_v}}$. We introduce certain quotients of 
$R_v^\square$:
\begin{itemize}
\item If $v \in S_p$, the smallest reduced $\cO$-torsion-free quotient $R_v$ of $R_v^\square$ such that if $F : R_v^\square \to \overline{\bQ}_p$ is a homomorphism such that the pushforward of the universal lifting to $\overline{\bQ}_p$ is crystalline of Hodge--Tate weights $\{0,1 \}$ (with respect to any embedding $F_v \to \overline{\bQ}_p$) and is not  ordinary, then $f$ factors through $R_v$. By \cite[Corollary 2.3.13]{MR2551765}, $R_v$ is a domain of dimension $4 + [F_v : \bQ_p]$.
\item If $v \in \Sigma$, the smallest reduced $\cO$-torsion-free quotient $R_v$ of $R_v^\square$ such that if $f : R_v^\square \to \overline{\bQ}_p$ is a homomorphism such that the pushforward of the universal lifting to $\overline{\bQ}_p$ is an extension of $\epsilon^{-1}$ by the trivial character, then $f$ factors through $R_v$. By \cite[Corollary 2.6.7]{Kis09}, $R_v$ is a domain of dimension 4.
\item If $v \in S_\infty$ (the set of infinite places  of $F$), the quotient $R_v$ of $R_v^\square$ denoted $R_{V_\bF}^{-1, \square}$ in \cite[Proposition 2.5.6]{MR2551765}. Then $R_v$ is a domain of dimension 3.
\end{itemize}
If $Q$ is a finite set of finite places of $F$, disjoint from $S$, then  we write $\operatorname{Def}_Q : \cC_\cO \to \operatorname{Sets}$ for the functor which associates to $A \in \cC_\cO$ the set of $1 + M_2(\ffrm_A)$-conjugacy classes of lifts $\widetilde{r} : G_F \to \GL_2(A)$ of $\overline{r}$ satisfying the following conditions:
\begin{itemize}
\item $\widetilde{r}$ is unramified outside $S \cup Q$, and $\det \widetilde{r} = \epsilon^{-1}$.
\item For each $v \in S \cup S_\infty$, the homomorphism $R_v^\square \to A$ determined by $\widetilde{r}|_{G_{F_v}}$ factors through the quotient $R_v^\square \to R_v$ introduced above.
\end{itemize}
Our assumption on the image of $\overline{r}$ implies that the functor $\operatorname{Def}_Q$ is represented by an object $R_Q \in \cC_\cO$. If $Q$ is empty then we write $R = R_\emptyset$. We also introduce some variants. Let $R_{loc} = \widehat{\otimes}_{v \in S  \cup S_\infty} R_v$. Then $R_{loc}$ is an $\cO$-flat domain of Krull dimension $1 + 3[F: \bQ] + 3 |S|$ (cf.~\cite[Lemma 3.4.12]{Kis09}). We write $\Def_Q^\square : \cC_\cO \to \operatorname{Sets}$ for the functor of $1+M_2(\ffrm_A)$-conjugacy of tuples $(\widetilde{r}, \{ A_v \}_{v \in S \cup S_\infty})$, where $\widetilde{r}$ is as above and $A_v \in 1 + M_2(\ffrm_A)$, and $\gamma \in 1+M_2(\ffrm_A)$ acts by $\gamma \cdot (\widetilde{r}, \{ A_v \}_{v \in S \cup S_\infty}) = (\gamma \widetilde{r} \gamma^{-1}, \{ \gamma A_v \}_{v \in S \cup S_\infty})$. This functor is represented by an object denoted $R_Q^\square \in \cC_\cO$. The tuple of representations $(A_v^{-1} \widetilde{r}|_{G_{F_v}} A_v)_{v \in S \cup S_\infty}$ is independent of the choice of  representative for a given conjugacy class, and the universal  property of $R_v^\square$ determines a homomorphism $R_{loc} \to R_Q^\square$.

The objects in this paragraph will only be used in the case $p = 2$. We write $\Def_Q' : \cC_\cO \to \operatorname{Sets}$ for the functor of $1 + M_2(\ffrm_A)$-conjugacy classes of lifts $\widetilde{r} : G_F \to \GL_2(A)$ of $\overline{r}$ satisfying the following conditions:
\begin{itemize}
\item $\widetilde{r}$ is unramified outside $S \cup Q$, and if $v \in S$ then $\det \widetilde{r}|_{G_{F_v}} = \epsilon^{-1}|_{G_{F_v}}$.
\item For each $v \in S \cup S_\infty$, the homomorphism $R_v^\square \to A$ determined by $\widetilde{r}|_{G_{F_v}}$ factors through the quotient $R_v^\square \to R_v$ introduced above.
\end{itemize}
We write $\Def_Q^{\prime, \square}$ for the functor of $1 + M_2(\ffrm_A)$-conjugacy classes of tuples $(\widetilde{r}, \{ A_v \}_{v \in S \cup S_\infty})$, where $\widetilde{r}$ is as above and $A_v \in 1 + M_2(\ffrm_A)$. Then the functors $\Def_Q'$ and $\Def_Q^{\prime, \square}$ are represented by objects $R_Q'$, $R_Q^{\prime, \square} \in \cC_\cO$ and there is again a natural morphism $R_{loc} \to R_Q^{\prime, \square}$.

Let $t = \det \Sym^{n-1} (r|_{G_K})$ and $t' = \det \Sym^{n-1} (r'|_{G_K})$ denote the group determinants over $\cO$ (in the sense of \cite{chenevier_det}) associated to these two symmetric power representations, and let $\overline{t}$ denote the group determinant over $k$ which is  their common reduction modulo $\varpi$. We introduce the object $P \in \cC_\cO$ which is the quotient $R_S$ of $R_{\overline{t}, S}^{[0, n-1]}$ introduced in \cite[\S 2.19]{New19a}. Informally, $P$ represents the functor of conjugate self-dual group determinants of $G_{K, S}$ lifting $\overline{t}$ which have similitude character $\epsilon^{1-n}$ and are semistable with Hodge--Tate weights in the interval $[0, n-1]$. 
\begin{lemma}\label{lem_sym_power_is_crystalline}
Let $A \in \cC_\cO$ be Artinian, let $v \in S_p$, and let $\widetilde{r} : G_{F_v} \to \GL_2(A)$ be a lift of $\overline{r}$ of determinant $\epsilon^{-1}|_{G_{F_v}}$ such that the associated homomorphism $R_v^\square \to A$ factors through $R_v^\square \to R_v$. View $A^n$ as an $A[G_{F_v}]$-module via the representation $\Sym^{n-1} \widetilde{r} : G_{F_v} \to \GL_n(A)$ (we equip $\Sym^{n-1}A^2$ with its standard ordered basis, cf.~\cite[Definition 3.3.1]{blgg}). Then $A^n$ is isomorphic, as $\bZ_p[G_{F_v}]$-module, to a subquotient of a lattice in a crystalline (in particular, semistable) $\bQ_p[G_{F_v}]$-module with all Hodge--Tate weights in the interval $[0, n-1]$.
\end{lemma}
\begin{proof}
It follows from the construction in \cite{Kis09, MR2551765} that the dual of $A^2$ is isomorphic, as $\bZ_p[G_{F_v}]$-module, to the generic fibre of a finite flat group scheme over $\cO_{F_v}$. By \cite[Th\'eor\`eme 3.1.1]{Ber82}, there exists a lattice $L$ in a crystalline $\bQ_p[G_{F_v}]$-representation with Hodge--Tate weights in the interval $[0, 1]$ such that $A^2$ is isomorphic, as $\bZ_p[G_{F_v}]$-module, to a quotient of $L$. It follows that $\Sym^{n-1}_A A^2$ is isomorphic, as $\bZ_p[G_{F_v}]$-module, to a quotient of $\Sym^{n-1}_{\bZ_p} L$. The proof is complete on noting that $\Sym^{n-1}_{\bZ_p} L$ is a lattice in $(\Sym^{n-1}_{\bZ_p} L)[1/p]$, a crystalline $\bQ_p[G_{F_v}]$-representation with all Hodge--Tate weights in the  interval  $[0, n-1]$.
\end{proof}
Let $A \in \cC_\cO$, and let $\widetilde{r} : G_F \to \GL_2(A)$ be a lift of $\overline{r}$ which determines a map $R \to A$. Lemma \ref{lem_sym_power_is_crystalline} shows that the pseudocharacter associated to $\Sym^{n-1} (\widetilde{r}|_{G_K})$ satisfies condition (2.16.1) of \cite{New19a}. In particular there are morphisms $P \to \cO$ associated to the pseudocharacters $t, t'$. Taking the pseudocharacter of the symmetric power of the universal deformation over $R$ determines a morphism $P \to R$ in $\cC_\cO$.  

We will study this morphism using the Taylor--Wiles method. In this paper we call a Taylor--Wiles datum of level $N \geq 1$ a tuple $(Q, \widetilde{Q}, (\alpha_v, \beta_v)_{v \in Q})$, where:
\begin{itemize}
\item $Q$ is a finite set of places of $F$, split in $K$, such that for each $v \in Q$, $q_v \equiv 1 \text{ mod }p^N$.
\item For each $v \in Q$, we have fixed a factorisation $v = \wv \wv^c$ in $K$ such that $\widetilde{Q} = \{ \wv \mid v \in Q \}$.
\item For each $v \in Q$, $\alpha_v, \beta_v \in k$ are eigenvalues of $\overline{r}(\Frob_v)$. We require that $\alpha_v^{n-1}, \alpha_v^{n-2} \beta_v, \dots, \beta_v^{n-1}$ are distinct elements of $k$.
\end{itemize}
We note that this last condition is stronger than the one typically appearing in applications to automorphy of 2-dimensional Galois representations and is specially adapted to our purposes here.

Let $r^{univ}_Q : G_F \to \GL_2(R_Q)$ denote a representative of the universal deformation.  If $v \in Q$ then $r^{univ}_Q|_{G_{F_v}}$ is conjugate (in  $\GL_2(R_Q)$) to a unique representation  of the  form $A_v \oplus B_v$, where $A_v \text{ mod }\ffrm_{R_Q} = \operatorname{ur}_{\alpha_v}$and $B_v \text{ mod }\ffrm_{R_Q} = \operatorname{ur}_{\beta_v}$. The characters $A_v, B_v : G_{F_v} \to R_Q^\times$ are independent of the choice of $r^{univ}_Q$. We write $\Delta_Q = \prod_{v \in Q} k(v)^\times(p)$ (i.e.  product of maximal $p$-power quotients of $k(v)^\times$). The collection of characters $A_v \circ \Art_{F_v}|_{\cO_{F_v}^\times}$ ($v \in Q$) determine a homomorphism $\cO[\Delta_Q] \to R_Q$ with the property that the natural map $R_Q \to R$ factors through a canonical isomorphism $R_Q \otimes_{\cO[\Delta_Q]} \cO \cong R$.

We write $P_{Q} \in \cC_\cO$ for the quotient of the quotient $R_{S \cup Q}$ of 
$R_{\overline{t}, S \cup Q}^{[0, n-1]}$ introduced in \cite[\S 2.19]{New19a} 
corresponding to pseudodeformations $\widetilde{t}$ of $\overline{t}$ with the 
following additional properties:
\begin{itemize}
	\item For each $v \in Q$, $\widetilde{t}|_{G_{K_\wv}}$ factors through the 
	maximal Hausdorff abelian quotient $G_{K_{\wv}} \to G_{K_\wv}^{ab}$.\footnote{It is possible that our condition on the eigenvalues of $\Frob_v$ necessarily entails that $\widetilde{t}$ is abelian at $\wv$. However, we haven't verified this and it doesn't cost us anything to build this in to the definition of $P_Q$.} By 
	\cite[Proposition 2.5]{All19} (cf.~\cite[Lemma 4.28]{New19a}, in \cite{All19} we generalise \cite[Proposition 1.5.1]{bellaiche_chenevier_pseudobook} to arbitrary characteristic), there are unique characters $\chi_{v, 1}, \dots, \chi_{v, n}$ lifting 
	$\operatorname{ur}_{\alpha_v^{n-i} \beta_v^{i-1}}$ such that 
	$\widetilde{t}|_{G_{K_\wv}}$ is the pseudocharacter associated to $\chi_{v, 
	1} \oplus \dots \oplus \chi_{v, n}$.
	\item Let $m = \lfloor n /2 \rfloor$ and let $i \in \{ 1, \dots, n \}$. If $n = 2m$ is even, then 
	$\chi_{v, i}|_{I_{K_\wv}} = \chi_{v, m}^{n+1-2i}|_{I_{K_\wv}}$. If $n = 
	2m+1$ is odd, then $\chi_{v, i}|_{I_{K_\wv}} = \chi_{v, 
	m}^{(n+1-2i)/2}|_{I_{K_\wv}}$. 
\end{itemize}
The characters $\chi_{v, m}$ ($v \in Q$) give $P_Q$ the structure of 
$\cO[\Delta_Q]$-algebra, and again there is a universal morphism $P_Q \to 
R_Q$. We remark that this need not be a morphism of $\cO[\Delta_Q]$-algebras when $n$ is odd, although it is when $n$  is even.

To carry out the patching argument, we  need to introduce spaces of automorphic forms. We first discuss automorphic forms on a definite quaternion algebra over $F$, following the set-up of \cite[\S 3.1]{MR2551765} and \cite[\S 7]{Kha09a}. Let $D$ be a definite quaternion algebra over $F$, ramified precisely at the infinite places and at the places of $\Sigma$. Fix a choice of maximal order $\cO_D$ and for each finite place $v \not\in \Sigma$, an identification $\cO_D \otimes_{\cO_F} \cO_{F_v} \cong M_2(\cO_{F_v})$. If $U = \prod_v U_v \subset (D \otimes_F \bA_F^\infty)^\times$ is an open subgroup, then we write $H_D(U)$ for the set of functions $f : (D \otimes_{\cO_F} \bA_F^\infty)^\times \to \cO$ such that for all $\gamma \in D^\times$, $z \in (\bA_F^\infty)^\times$, $g \in (D \otimes_{\cO_F} \bA_F^\infty)^\times$, and $u \in U$, we have $f(\gamma g z u) = f(g)$. 

We define 
\[ U_0 = \left( \prod_{v \not\in \Sigma} (\cO_D \otimes_{\cO_F} \cO_{F_v})^\times \right) \times \left( \prod_{v \in \Sigma} (D \otimes_F F_v)^\times \right). \]
If $(Q, \widetilde{Q}, \{ \alpha_v, \beta_v \}_{v \in Q})$ is a Taylor--Wiles datum of level $N \geq 1$, then we define $U_1(Q;N) = \prod_v U_1(Q; N)_v \subset U_0(Q) = \prod_v U_0(Q)_v$ by $U_1(Q; N)_v = U_0(Q)_v = U_{0, v}$ if $v \not\in Q$, and $U_0(Q)_v = \Iw_v$ and 
\[ U_1(Q; N)_v = \left\{\left( \begin{array}{cc} a & b \\ c & d \end{array} \right) \in \Iw_v : a d^{-1} \mapsto 1 \in k(v)^\times(p) /  (p^N) \right\} \]
if $v \in Q$. Thus $U_1(Q;N) \subset U_0(Q)$ is a normal subgroup with quotient  
\[ U_0(Q) / U_1(Q; N) \cong \Delta_Q / (p^N). \]

We introduce Hecke operators. If $v \not\in \Sigma \cup Q$ is a finite place of $F$, then the 
unramified Hecke operators $T_v^{(1)}, T_v^{(2)}$  act on $H_D(U_0(Q))$ and 
$H_D(U_1(Q; N))$. If $v \in Q$ then the operator $U^{(1)}_{\varpi_v}$ acts on 
$H_D(U_0(Q))$ and $H_D(U_1(Q; N))$.  We write $\bT^{univ}_{D, \Sigma \cup Q}$ 
for the polynomial ring over $\cO$ in the indeterminates $T_v^{(1)}, 
T_v^{(2)}$($v \not\in \Sigma \cup Q$) and $\bT^{univ, Q}_{D, \Sigma \cup Q} = 
\bT^{univ}_{D, \Sigma \cup Q}[\{U^{(1)}_{\varpi_v}\}_{v \in Q}]$.

There is a unique maximal ideal $\ffrm_D \subset \bT^{univ}_{D, \Sigma}$ of 
residue field $k$ such that for all finite places $v \not\in S$ of $F$, 
the characteristic polynomial of $\overline{r}(\Frob_v)$ equals $X^2 - 
T_v^{(1)} X + q_v T_v^{(2)} \text{ mod }\ffrm_D$ and for each $v \in S_p$, $T_v^{(1)} \in \ffrm_D$ and $T_v^{(2)} - 1 \in \ffrm_D$. If $(Q, \widetilde{Q}, 
(\alpha_v, \beta_v)_{v \in Q})$ is a Taylor--Wiles datum, then we  write 
$\ffrm_{D, Q}$ for the maximal ideal of $\bT^{univ, Q}_{D, \Sigma \cup Q}$ 
generated by $\ffrm_D \cap \bT^{univ}_{D, \Sigma \cup Q}$ and the elements 
$U^{(1)}_{\varpi_v} - \alpha_v$ ($v \in Q$). 

If  $\chi : F^\times \backslash (\bA_F^\infty)^\times / \det U_1(Q; N) \to \cO^\times$ is a quadratic character  and $f \in  H_D(U_1(Q; N))$, then we define  $f \otimes \chi \in  H_D(U_1(Q; N))$  by the formula $(f \otimes \chi)(g) = \chi(\det(g)) f(g)$. We observe that if $p = 2$ and $f  \in H_D(U_1(Q; N))_{\ffrm_{D, Q}}$ then $f \otimes \chi \in H_D(U_1(Q; N))_{\ffrm_{D, Q}}$.

Let $m_D \geq 0$ denote the $p$-adic valuation of the least common multiple of 
the exponents of the Sylow $p$-subgroups of the finite groups $F^\times 
\backslash (U (\bA_F^\infty)^\times \cap t^{-1} D^\times t)$ for $t \in (D 
\otimes_F \bA_F^\infty)^\times$; this number is finite, see \cite[\S 
7.2]{Kha09a}.
\begin{proposition}\label{prop_automorphic_forms_on_D_at_level_Q} Let $N \geq 1$ and let $(Q, \widetilde{Q}, (\alpha_v, \beta_v)_{v \in Q})$ be a Taylor--Wiles datum of level $N + m_D$.
\begin{enumerate}
\item The maximal ideals $\ffrm_D$ and $\ffrm_{D, Q}$ are in the support of $H_D(U_0)$ and $H_D(U_0(Q))$, respectively.
\item $H_D(U_1(Q; N))_{\ffrm_{D, Q}}$ is a $\bT^{univ}_{D, \Sigma \cup Q} \otimes_\cO \cO[\Delta_Q / (p^N)]$-module free as $\cO[\Delta_Q / (p^N)]$-module, and there is an isomorphism $H_D(U_1(Q; N))_{\ffrm_{D, Q}} \otimes_{\cO[\Delta_Q / (p^N)]} \cO \cong H_D(U_0)_{\ffrm_D}$ of $\bT^{univ}_{D, \Sigma \cup Q}$-modules.
\item There exists a structure on $H_D(U_1(Q; N))_{\ffrm_{D, Q}}$ of 
$R_Q$-module such that for any representative $r_Q^{univ}$ of the universal 
deformation of $\overline{r}$ and for each finite place $v\not\in S \cup Q$ of 
$F$, $\tr r_Q^{univ}(\Frob_v)$ acts as $T_v^{(1)}$ and $\det 
r_Q^{univ}(\Frob_v)$ acts as $q_v T_v^{(2)}$. Moreover, the 
$\cO[\Delta_Q]$-module structure induced by the map $\cO[\Delta_Q] \to R_Q$ 
agrees with the one in the second part of the lemma.
\end{enumerate}
\end{proposition}
\begin{proof}
The first part is a consequence of the Jacquet--Langlands correspondence (and the existence of $\pi$). The second part is \cite[Corollary 7.5]{Kha09a}. The third part is proved in the same way as \cite[Lemma 9.1]{Kha09a} (cf.~also \cite[Lemma 3.2.7]{MR2551765}); note that since $T_v^{(1)} \in \ffrm_D$ for each $v \in S_p$, only automorphic representations which are non-ordinary at each $v \in S_p$ can contribute to $H_D(U_1(Q; N))_{\ffrm_D}$.
\end{proof}
We next discuss automorphic forms on a definite unitary group of  rank $n$. We therefore fix a unitary group $G$ over $F$, split by $K / F$, as in \cite[\S 4.1]{New19a}, together with an extension of $G$ to a reductive group scheme over $\cO_F$. We recall that $G$ comes equipped with isomorphisms $\iota_w : G_{\cO_{F_v}} \to \Res_{\cO_{K_w} / \cO_{F_v}} \GL_n$ for each place $v$ of $F$ which splits $v = w w^c$ in $K$. Moreover, for each place $v \nmid \infty$ of $F$, $G_{F_v}$ is quasi-split, while for each place $v | \infty$ of $F$, $G(F_v)$ is compact.

If $V = \prod_v V_v \subset G(\bA_{F^+}^\infty)$ is an open compact subgroup, then we write $H_G(V)$ for the set of functions $f : G(F) \backslash G(\bA_{F}^\infty) / V \to \cO$. We define $V_0 = \prod_v V_{0, v}$ by choosing $V_{0, v} = G(\cO_{F_v})$ if $v \not\in \Sigma$ and $V_{0, v} = \iota_\wv^{-1} \Iw_\wv$ if $v \in \Sigma$. If $(Q, \widetilde{Q}, \{ \alpha_v, \beta_v \}_{v \in Q})$ is a Taylor--Wiles datum of level $N \geq 1$, then we define $V_1(Q; N) = \prod_v V_1(Q; N)_v \subset V_0(Q) = \prod_v V_0(Q)_v$ by $V_1(Q; N)_v = V_0(Q)_v = V_{0, v}$ if $v \not\in Q$, and $V_0(Q)_v = \iota_\wv^{-1} \Iw_\wv$ and
\[ V_1(Q; N)_v = \iota_\wv^{-1} \{ (a_{ij}) \in \Iw_\wv \mid \prod_{i=1}^n a_{ii}^{n+1-2i} \mapsto 1 \in k(\wv)^\times(p) / (p^N) \}\]
if $v \in Q$ and $n$ is even, and 
\[ V_1(Q; N)_v = \iota_\wv^{-1} \{ (a_{ij}) \in \Iw_\wv \mid \prod_{i=1}^n a_{ii}^{(n+1-2i)/2} \mapsto 1 \in k(\wv)^\times(p) / (p^N) \}\]
if $v \in Q$ and $n$ is odd. Thus $V_1(Q; N) \subset V_0(Q)$ is a normal subgroup with quotient  \[ V_0(Q) / V_1(Q; N) \cong \Delta_Q  / (p^N). \]
 We introduce Hecke operators for the group $G$. If $v \not\in S \cup Q$ is a place of $F$ which splits $v = w w^c$ in $K$, then the unramified Hecke operators $\iota_w^{-1} T_w^{(i)}$ ($i = 1, \dots n$) act on the spaces $H_G(V_0(Q))$ and $H_G(V_1(Q; N))$. If $v \in Q$ then the operators $\iota_\wv^{-1} U_{\varpi_\wv}^{(i)}$ ($i = 1, \dots, n$) act on the spaces $H_G(V_0(Q))$ and $H_G(V_1(Q; N))$. We  write $\bT^{univ}_{G, S \cup Q}$ for the polynomial ring over $\cO$ in the indeterminates $T_w^{(i)}$ (where $w$  is a place of $K$ split over  $F$, not lying above  $S \cup Q$ and $1 \leq i \leq n$), and $\bT^{univ, Q}_{G,S \cup  Q} = \bT^{univ}_{G, S\cup Q}[\{ U_{\varpi_\wv}^{(i)} \}^{i = 1, \dots, n}_{v \in Q}]$. Thus $H_G(V_0(Q))$ and $H_G(V_1(Q; N))$ are $\bT^{univ, Q}_{G, S\cup Q}$-modules.

There is a unique maximal ideal $\ffrm_G \subset \bT^{univ}_{G, S}$ of residue field $k$ such that for all finite places $w$ of $K$, split over $F$ and not lying above $S$, the characteristic polynomial of $\Sym^{n-1} \overline{r} (\Frob_w)$ equals $\sum_{i=0}^n (-1)^i q_w^{i(i-1)/2 } T_w^{(i)} X^{n-i} \text{ mod }\ffrm_G$. If $(Q, \widetilde{Q}, (\alpha_v, \beta_v)_{v \in Q})$ is a Taylor--Wiles datum, then we write $\ffrm_{G, Q}$ for the maximal ideal of $\bT^{univ, Q}_{G, S \cup Q}$ generated by $\ffrm_G \cap \bT^{univ}_{G, S \cup Q}$ and the elements 
\[ U_{\varpi_\wv}^{(i)} - q_\wv^{i(1-i)/2} \prod_{j=1}^i \alpha_v^{n-j} \beta_v^{j-1} \]
for $v \in Q$, $i = 1, \dots, n$.

The unitary group $G$ comes with a determinant map $\det : G \to U_1$, where $U_1  = \ker (\mathbf{N}_{K / F} : \Res_{K / F} \bG_m \to \bG_m)$. If
\[ \theta : U_1(F) \backslash U_1(\bA_F^\infty) / \det V_1(Q;  N) \to \cO^\times\]
 is a character and $f \in H_G(V_1(Q; N))$ then we define $f \otimes \theta  \in H_G(V_1(Q; N))$ by  the formula $(f \otimes \theta)(g)  = \theta(\det(g)) f(g)$.   If $f \in H_G(V_1(Q;  N))_{\ffrm_{G, Q}}$ and $\overline{\theta}$ is trivial, then  $f \otimes  \theta \in H_G(V_1(Q;  N))_{\ffrm_{G, Q}}$. We will use this construction only in conjunction with the following lemma.
 \begin{lemma}
 Suppose that $n$ is even and that $p = 2$. Suppose that  
 \[ \chi : F^\times \backslash (\bA_F^\infty)^\times / \det U_1(Q; N) \to \cO^\times
\]
 is a quadratic character. Then there exists a unique character \[ \theta_\chi  : U_1(F) \backslash U_1(\bA_F^\infty) / \det V_1(Q;  N) \to  \cO^\times \]
such that for all $z \in (\bA_K^\infty)^\times$ we have $\theta_\chi(z / z^c) = 
\chi(z  z^c)$. 
 \end{lemma} 
 \begin{proof}
 There is a short exact sequence of $F$-groups
 \[ 0 \to \bG_m \to \Res_{K / F} \bG_m \to U_1 \to 0, \]
where the last map is $z \mapsto z / z^c$. By Hilbert 90 we have a short exact sequence
\[ 0 \to F^\times \backslash (\bA_F^\infty)^\times \to K^\times \backslash (\bA_K^\infty)^\times \to U_1(F) \backslash U_1(\bA_F^\infty) \to 0. \]
This shows that there is a unique character $\theta_\chi : U_1(F) \backslash U_1(\bA_F^\infty)  \to \cO^\times$ such that $\theta_\chi(z / z^c) = \chi(z z^c)$ for all $z \in (\bA_K^\infty)^\times$. We need to check that $\theta_\chi$ is trivial on $\det V_1(Q;  N)$. This can be checked locally at each finite place of $F$. At places $v \not\in Q$ it follows from the fact that $\chi$ is unramified at $v$. If $v \in Q$ then we see, using that $n$ is even and identifying $U_1(F_v)$ with $F_\wv^\times$, that $\det V_1(Q; N)_v$ is contained in the subgroup of $\cO_{F_\wv}^\times$ with square image in $k(\wv)^\times$. Since $\chi$ is quadratic, $\theta_\chi$ annihilates this subgroup, and we're done. 
 \end{proof}
Let $m_G \geq 0$ denote the $p$-adic valuation of the least common multiple of 
the exponents of the Sylow $p$-subgroups of the finite groups $G(F) \cap t V_0 
t^{-1}$ ($t \in G(\A_F^\infty)$). 
\begin{proposition}\label{prop_automorphic_forms_on_G_at_level_Q}
Let $N \geq 1$ and let $(Q, \widetilde{Q}, (\alpha_v, \beta_v)_{v \in Q}))$ be a Taylor--Wiles datum of level $N + m_G$. Then:
\begin{enumerate}
\item The  maximal ideals $\ffrm_G$, $\ffrm_{G, Q}$ are  in the support of $H_G(V_0)$ and $H_G(V_0(Q))$, respectively.
\item $H_G(V_1(Q; N))_{\ffrm_{G, Q}}$ is a $\bT^{univ}_{G, S \cup Q} \otimes_\cO \cO[\Delta_Q / (p^N)]$-module free as $\cO[\Delta_Q / (p^N)]$-module, and there is an isomorphism $H_G(V_1(Q; N))_{\ffrm_{G, Q}} \otimes_{\cO[\Delta_Q / (p^N)]} \cO \cong H_G(V_0)_{\ffrm_G}$ of $\bT^{univ}_{G, S \cup Q}$-modules.
\item There exists a structure on $H_G(V_1(Q; N))_{\ffrm_{G, Q}}$ of $P_Q$-module such that if $\Lambda_i^{univ} : G_K \to P_Q$ ($i = 1, \dots, n$) are the coefficients of the universal characteristic polynomial, defined as in \cite[\S 1.10]{chenevier_det}, then for any finite place $v \not\in S \cup Q$ of $F$ which splits $v = w w^c$ in $K$, $\Lambda_i^{univ}(\Frob_w)$ acts on $H_G(V_1(Q; N))_{\ffrm_{G, Q}}$ as $q_w^{i(i-1)/2} T_w^{(i)}$. Moreover, the $\cO[\Delta_Q]$-module structure on  $H_G(V_1(Q; N))_{\ffrm_{G, Q}}$ induced by the map $\cO[\Delta_Q] \to P_Q$ agrees with the one in the second part of the lemma.
\end{enumerate}
\end{proposition}
\begin{proof}
The first part may be deduced, as in \cite[\S 4.3]{New19a}, from \cite[Th\'eor\`eme 5.4]{labesse} (and the existence of $\Pi'_n$). The other parts are proved in a very similar way to the second and third parts of Proposition \ref{prop_automorphic_forms_on_D_at_level_Q}, as we now explain. We begin by constructing a more familiar set of objects. Let $\Delta'_Q = \prod_{v \in Q} k(v)^\times(p)^n$. Let $V'_1(Q; N) = \prod_v V'_1(Q; N)_v \subset V_0(Q)$ be  the subgroup defined by $V'_1(Q; N)_v = V_0(Q)_v$ if $v \not\in Q$ and 
\[ V'_1(Q; N)_v = \iota_\wv^{-1} \{ (a_{ij}) \in \Iw_\wv \mid \forall i = 1, \dots, n, a_{ii} \mapsto 1 \in k(\wv)^\times(p) / (p^N) \} \]
if $v \in Q$. Thus $V'_1(Q; N) \subset V_0(Q)$ is a normal subgroup with 
quotient $V_0(Q) / V'_1(Q; N) \cong \Delta_Q' / (p^N)$. We write $P_Q' \in 
\cC_\cO$ for the quotient of the ring $R_{S \cup Q}$ introduced in \cite[\S 2.19]{New19a} corresponding to pseudodeformations whose restriction to $G_{K_\wv}$ factors through $G_{K_{\wv}}^{ab}$ for each $v \in Q$. As in the case of $P_Q$, \cite[Proposition 2.5]{All19} again shows that for each $v \in Q$ there are 
unique characters $A_v^{(i)} : G_{K_\wv} \to (P_Q')^\times$ ($i = 1, \dots, n$) 
such that $A_v^{(i)} \text{ mod }\ffrm_{P'_Q} = 
\operatorname{ur}_{\alpha_v^{n-i} \beta_v^{i-1}}$ and $\det(A_v^{(1)} 
\oplus \dots \oplus A_v^{(n)})$ is the restriction to $G_{K_\wv}$ of 
the universal pseudocharacter over $G_{K}$. We now claim that the following 
statements hold:
\begin{itemize}
\item $H_G(V'_1(Q; N))_{\ffrm_{G, Q}}$ is a $\bT^{univ}_{G, S \cup Q} \otimes_\cO \cO[\Delta'_Q / (p^N)]$-module free as $\cO[\Delta'_Q / (p^N)]$-module, and there is an isomorphism $H_G(V'_1(Q; N))_{\ffrm_{G, Q}} \otimes_{\cO[\Delta'_Q / (p^N)]} \cO \cong H_G(V_0)_{\ffrm_G}$ of $\bT^{univ}_{G,S\cup Q}$-modules.
\item There is a unique structure on $H_G(V'_1(Q; N))_{\ffrm_{G, Q}}$ of $P'_Q$-module such that for any place $v \not\in S \cup Q$ of $F$ which splits $v = w w^c$ in $K$, $\Lambda_i^{univ}(\Frob_\wv)$ acts on $H_G(V'_1(Q; N))_{\ffrm_{G, Q}}$ as $q_w^{i(i-1)/2} T_w^{(i)}$.
\item The two induced $\cO[\Delta'_Q]$-module structures on $H_G(V'_1(Q; N))_{\ffrm_{G, Q}}$ (one by the isomorphism $V_0(Q) / V_1'(Q; N) \cong \Delta'_Q / (p^N)$, the other by the map $\Delta'_Q \to P'_Q$ associated to the tuple of characters $A_v^{(i)} \circ \Art_{K_\wv}|_{\cO_{K_\wv}^\times}$) are the same.
\end{itemize}
To prove the first point, we need to explain why $H_G(V'_1(Q; N))_{\ffrm_{G, 
Q}}$ is an $\cO[\Delta'_Q / (p^N)]$-module with the claimed coinvariants. The 
action of $\Delta'_Q / (p^N)$ is induced by the action of $V_0(Q)$ via the 
isomorphism $V_0(Q) / V_1'(Q; N) \cong \Delta'_Q / (p^N)$ (which therefore 
commutes with the action of $\bT^{univ}_{G, S \cup Q}$). The freeness 
follows because $V_0(Q) / V_1'(Q; N)$ acts freely on the quotient $G(F) 
\backslash G(\bA_F^\infty) / V_1'(Q; N)$, because $V_1'(Q; N)$ contains all the 
$p$-torsion elements of $V_0(Q)$; compare the proof of \cite[Lemma 7.4]{Kha09a} 
and \cite[Lemma 8.18]{Boc19}. The freeness of this action implies that there is 
an isomorphism $H_G(V'_1(Q; N))_{\ffrm_{G, Q}} \otimes_{\cO[\Delta'_Q / (p^N)]} 
\cO \cong H_G(V_0(Q))_{\ffrm_{G, Q}}$ of $\bT^{univ}_{G, S \cup 
Q}$-modules. To complete the proof of the first point, we need to explain why 
there is an isomorphism $H_G(V_0(Q))_{\ffrm_{G, Q}} \cong H_G(V_0)_{\ffrm_G}$. 
This follows from \cite[Proposition 3.1]{New19a}. The second part is proved in 
the same way as \cite[Lemma 4.7]{New19a}.  The third part is proved using  
\cite[Lemma 4.7]{New19a} and \cite[Proposition 
1.5.1]{bellaiche_chenevier_pseudobook}  (in particular, the uniqueness of  the 
decomposition of residually multiplicity-free pseudocharacters). 

We now need to explain why the above claims imply the properties in the statement of the proposition. There are canonical quotient morphisms $P'_Q \to P_Q$ and $\Delta'_Q \to \Delta_Q$. The proof is complete on noting that trace induces an isomorphism $H_G(V'_1(Q; N))_{\ffrm_{G, Q}} \otimes_{\cO[\Delta'_Q] / (p^N)} \cO[\Delta_Q / (p^N)] \cong H_G(V_1(Q; N))_{\ffrm_{G, Q}}$ and that the map $P'_Q \to P_Q$ factors through an isomorphism  $P'_Q \otimes_{\cO[\Delta'_Q]} \cO[\Delta_Q] \cong P_Q$.
\end{proof}
We are now ready to prove Theorem \ref{thm_automorphy_of_symmetric_power}. We will need to treat the cases $p > 2$ and $p = 2$ separately.
\begin{proof}[Proof of Theorem \ref{thm_automorphy_of_symmetric_power}, case $p > 2$]
Define $H_G = H_G(V_0)_{\ffrm_G}$ and $H_D = H_D(U_0)_{\ffrm_D}$.	The proof of the theorem will be based on the following proposition:
\begin{proposition}\label{prop_construction_of_patched_data_p_odd}
We can find an integer $q \geq 0$ with the following property: let $W_\infty = \cO \llbracket Y_1, \dots, Y_q, Z_1, \dots, Z_{4 |S \cup S_\infty| -1 } \rrbracket$. Then we can find the following data:
\begin{enumerate}
\item Complete Noetherian local $W_\infty$-algebras $P_\infty$, $R_\infty$ equipped with isomorphisms $P_\infty \otimes_{W_\infty} \cO \cong P$ and $R_\infty \otimes_{W_\infty} \cO \cong R$ in $\cC_\cO$.
\item A surjection $R_{loc} \llbracket X_1, \dots, X_g \rrbracket \to R_\infty$ in $\cC_\cO$, where $g = q + |S \cup S_\infty| - 1$.
\item A $P_\infty$-module $H_{G,\infty}$ and an $R_\infty$-module $H_{D,\infty}$ such that both $H_{G,\infty}$, $H_{D,\infty}$ are finite free $W_\infty$-modules, complete with isomorphisms $H_{G,\infty} \otimes_{W_\infty} \cO \cong H_G$ (as $P$-module) and $H_{D,\infty} \otimes_{W_\infty} \cO \cong H_D$ (as $R$-module).
\item A morphism $P_\infty \to R_\infty$ of $\cO$-algebras making the diagram
\[ \xymatrix{ P_\infty \ar[r] \ar[d] &  R_\infty \ar[d] \\
P \ar[r] & R } \]
commute.
\end{enumerate}
\end{proposition}
Before giving the proof of Proposition \ref{prop_construction_of_patched_data_p_odd}, we show how it implies the theorem. Let $\p \subset P$ denote the kernel of the morphism $P \to \cO$ associated to $t$. It is enough to show that $\p$ is in the support of $H_G$ as $P$-module. Indeed, this would imply (using \cite[Corollaire 5.3]{labesse} and the irreducibility of $\Sym^{n-1} r|_{G_K}$) the existence of a RACSDC automorphic representation $\Pi_n$ of $\GL_n(\A_K)$ such that $r_{\Pi_n, \iota} \cong \Sym^{n-1} (r|_{G_K})$. By descent (in the form of e.g.\ \cite[Lemma 1.5]{cy2}), this would imply the sought-after automorphy of $\Sym^{n-1} r$. Equivalently, we must show that if $\p_\infty$ is the pre-image of $\p$ under the morphism $P_\infty \to P$, then $\p_\infty$ is in the support of $H_{G,\infty}$ as $P_\infty$-module. (Since $\Supp_P H_G = \Supp_{P_\infty} H_{G,\infty} \cap \Spec P$, intersection taken inside $\Spec P_\infty$.)

The $P_\infty$-module $H_{G,\infty}$ is a Cohen--Macaulay module (i.e.\ it is finite and the dimension of its support is equal to its depth), since $H_{G,\infty}$ is a finite free $W_\infty$-module. Applying \cite[\href{https://stacks.math.columbia.edu/tag/0BUS}{Tag 0BUS}]{stacks-project}, it follows that each irreducible component of $\Supp_{P_\infty} H_{G,\infty}$ has dimension $q + 4|S \cup S_\infty|$. Similarly, we see that each irreducible component of $\Supp_{R_\infty}  H_{D,\infty}$ has dimension $q + 4|S \cup S_\infty|$. Since $R_\infty$ is a quotient of $R_{loc} \llbracket X_1, \dots, X_g \rrbracket $, a domain of Krull dimension $q + 4|S \cup S_\infty|$, we see that $R_{loc} \llbracket X_1, \dots, X_g \rrbracket \to R_\infty$ is an isomorphism, that $R_\infty$ is a domain, and that $H_{D,\infty}$ is a faithful $R_\infty$-module.

Let $\p' \subset P$ denote the kernel of the morphism $P \to \cO$ associated to $t'$, and let $\p'_\infty$ denote the pre-image of $\p'$ under the morphism $P_\infty \to P$. Then $\p'_\infty \in \Supp_{P_\infty} H_{G,\infty}$, by hypothesis, and therefore $\dim(P_{\infty, (\p'_\infty)}) \ge q + 4|S \cup S_\infty|-1$. We claim that the Zariski tangent space to the local ring $P_{\infty, (\p'_\infty)}$ has dimension at most $q + 4|S \cup S_\infty|-1$. Indeed, it suffices to note that the quotient $P_{\infty, (\p'_\infty)} / (Y_1, \dots, Y_q, Z_1, \dots, Z_{4 |S \cup S_\infty|-1}) = P_{(\p')}$ equals its residue field $E$, by the vanishing of the adjoint Bloch--Kato Selmer group of $r_{\Pi_n',\iota}$ (i.e.\ by \cite[Theorem A, Proposition 2.21, Example 2.34]{New19a}). We deduce that $P_{\infty, (\p'_\infty)}$ is a regular local ring of dimension $q + 4|S \cup S_\infty|-1$, so there is a unique irreducible component $Z$ of $\Spec P_\infty$ containing the point $\p'_\infty$, which has dimension $q + 4|S \cup S_\infty|$ and is contained in $\Supp_{P_\infty} H_{G,\infty}$.

Since $\Spec R_\infty$ is irreducible and the image of the morphism $\Spec R_\infty \to \Spec P_\infty$ contains $\p'_\infty$, we find that the morphism $\Spec R_\infty \to \Spec P_\infty$ factors through $Z$. In particular, $\p_\infty$ lies in $Z$, hence in $\Supp_{P_\infty} H_{G,\infty}$. This completes the proof of the theorem.

The proof of Proposition \ref{prop_construction_of_patched_data_p_odd} is based on a patching argument. We first prove a lemma which shows that there are enough Taylor--Wiles data. The argument is very similar (and essentially identical in the case $n = 2$) to the proof of \cite[Theorem 2.49]{DDT}. We spell out the details here just to show that the condition that the numbers $\alpha_v^{n-i} \beta_v^{i-1}$ ($i = 1, \dots, n$) are distinct does not cause any difficulty.
\begin{lemma}\label{lem_existence_of_TW_data_p_neq_2}
Let $q = \dim_k H^1(F_S  / F, \ad^0 \overline{r}(1))$, and let $g = q + |S \cup S_\infty|-1$. Then for any $N \geq 1$, we can find a Taylor--Wiles datum $(Q, \widetilde{Q}, (\alpha_v, \beta_v)_{v \in Q})$ of level $N  \geq 1$ such that there is a surjection $R_{loc} \llbracket X_1, \dots, X_g \rrbracket \to R_Q^{\square}$ of  $R_{loc}$-algebras.
\end{lemma}
\begin{proof}
	Let $(Q, \widetilde{Q}, (\alpha_v, \beta_v)_{v \in Q})$ be a Taylor--Wiles datum. A standard computation (compare e.g.\ \cite[Proposition 3.2.5]{Kis09}, \cite[Proposition 5.10]{Tho16}), shows that there is a surjection $R_{loc} \llbracket X_1, \dots, X_g \rrbracket \to R_Q^{\square}$ where $g = \lambda_Q + |Q| + |S \cup S_\infty| - 1$  and $\lambda_Q$ is the dimension of the group
	\[ \ker\left(H^1(F_S / F, \ad \overline{r}^0(1)) \to \prod_{v \in Q} H^1(F_v, \ad^0 \overline{r}(1)) \right). \]
	We therefore need to show that for any $N \geq 1$, we can find a Taylor--Wiles datum of level $N$ such that $|Q| = q$ and $\lambda_Q = 0$. By induction, and the Chebotarev density theorem, it is enough to show the following claim:
	\begin{itemize}
		\item Let $[\phi] \in H^1(F_S / F, \ad^0 \overline{r}(1))$ be non-zero. Then for any $N \geq 1$, there exists $\sigma \in G_{F(\zeta_{p^N})}$ such that $\phi(\sigma) \not\in (\sigma - 1) \ad^0 \overline{r}(1)$ and the eigenvalues $\alpha,\beta$ of $\rbar(\sigma)$ satisfy $(\alpha/\beta)^i \neq 1$ for $i = 1,\ldots,n-1$.
	\end{itemize}
	Let $L / F$ denote the extension cut out by $\operatorname{Proj} \overline{r}$. Recall our assumption that $\Gamma = \Gal(L / F)$ is conjugate in $\PGL_2(\overline{\bF}_p)$ either to $\PSL_2({\bF}_{p^a})$ or $\PGL_2(\bF_{p^a})$ for some $p^a > \max(5, 2n-1)$. In either case $[\Gamma, \Gamma] \cong\PSL_2({\bF}_{p^a})$ is a non-abelian simple group that acts absolutely irreducible on $\ad^0$. Moreover, \cite[Table (4.5)]{Cli75} shows that $H^1(\PSL_2({\bF}_{p^a}), \ad^0) = 0$. The inflation-restriction exact sequence implies that $H^1(L(\zeta_{p^N}) / F, \ad^0 \overline{r}(1)) = 0$. Another application of inflation-restriction shows that $\Res_{L(\zeta_{p^N}) / F} [\phi] \neq 0$; we may identify this restriction with a non-zero $G_F$-equivariant homomorphism $f : G_{L(\zeta_{p^N})} \to \ad^0 \overline{r}(1)$. Since $\ad^0 \overline{r}$ is absolutely irreducible as a $k[\PSL_2(\bF_{p^a})]$-module, $f(G_{L(\zeta_{p^N})})$ spans $\ad^0 \overline{r}(1)$ as $k$-vector space.
	
	To prove the claim above, choose any $\sigma \in G_{F(\zeta_{p^N})}$ such that the eigenvalues $\alpha, \beta \in k$ of $\overline{r}(\sigma)$ satisfy $(\alpha / \beta)^{i} \neq 1$ for $i = 1, \dots, n-1$. (This is possible since if $t$ generates $\bF_{p^a}^\times$, then $t^2$ has multiplicative order $\geq (p^a - 1) / 2 > n-1$ and $\operatorname{Proj} \overline{r}(G_{F(\zeta_{p^N})})$ contains an element which is conjugate to $\diag(t, t^{-1}) = \diag(t^2, 1)$.)
	
	If $\phi(\sigma) \not\in (\sigma-1) \ad^0 \overline{r}(1)$ then we're done. Suppose instead that $\phi(\sigma) \in (\sigma-1) \ad^0 \overline{r}(1)$. Since $f(G_{L(\zeta_{p^N})})$ spans $\ad^0 \overline{r}(1)$, we can find $\tau \in G_{L(\zeta_{p^N})}$ such that $f(\tau) \not\in (\sigma-1) \ad^0 \overline{r}(1)$. Then $\tau \sigma \in G_{F(\zeta_{p^N})}$, $\operatorname{Proj} \overline{r}(\tau \sigma) = \operatorname{Proj} \overline{r}(\sigma)$, and the cocycle relation shows that 
	\[ \phi(\tau \sigma) = \phi(\tau) + \phi(\sigma) \not\in (\tau \sigma-1) \ad^0 \overline{r}(1). \]
	In either case we have established the claim; this completes the proof.
\end{proof}
Now we give the proof of Proposition \ref{prop_construction_of_patched_data_p_odd}. Let $q = \dim_k H^1(F_S  / F, \ad^0 \overline{r}(1))$ and $g = q + |S \cup S_\infty| - 1$. For each $N \geq 1$, fix a Taylor--Wiles datum $(Q_N, \widetilde{Q}_N, (\alpha_v, \beta_v)_{v \in Q_N})$ of level $N + \max(m_D, m_G)$ such that there exists a surjection $R_{loc}\llbracket X_1, \dots, X_g \rrbracket \to R_{Q_N}^\square$ of $R_{loc}$-algebras. Fix $v_0 \in S$ and define $\cT = \cO \llbracket \{ Z_{v, i, j} \}_{v \in S \cup S_\infty, 1 \leq i, j \leq 2} \rrbracket / (Z_{v_0, 1, 1})$. We view $\cT$ as an augmented $\cO$-algebra via the augmentation which sends each $Z_{v, i, j}$ to $0$. Choose for each $N \geq 1$ a surjection $\cO \llbracket Y_1, \dots, Y_q \rrbracket \to \cO[\Delta_{Q_N} / (p^N)]$. Then we get surjections $W_\infty  \to \cT \widehat{\otimes}_\cO \cO[\Delta_{Q_N}/(p^N)] \to \cO[\Delta_{Q_N} / (p^N)]$.

Define $R_N = R^\square_{Q_N}$, $H_{D,N} = \cT \widehat{\otimes}_\cO H_D(U_1(Q; N))_{\ffrm_{D,Q_N}}$, $P_N = \cT \widehat{\otimes}_\cO P_{Q_N}$, $H_{G,N} = \cT \widehat{\otimes}_\cO H(V_1(Q; N))_{\ffrm_{G, Q_N}}$.  We fix a representative $r^{univ} : G_F \to \GL_2(R)$ for the universal deformation over $R$, and representatives $r_{Q_N}^{univ} : G_F \to \GL_2(R_{Q_N})$ for the universal deformations over $R_{Q_N}$ lifting $r^{univ}$ for each $N \geq 1$. These choices determine isomorphisms $R^\square \cong \cT \widehat{\otimes}_\cO R$ and $R_{N} \cong \cT \widehat{\otimes}_\cO R_{Q_N}$, which classify the universal $S \cup S_\infty$-framed liftings $(r^{univ}, \{ 1 + (Z_{v, i,j}) \}_{v \in S \cup S_\infty})$ (resp. $(r_{Q_N}^{univ}, \{ 1 + (Z_{v, i,j}) \}_{v \in S \cup S_\infty})$). Thus each ring $R_N, P_N$ has a $W_\infty$-algebra structure, and there are isomorphisms $R_N \otimes_{W_\infty} \cO \cong R$, $P_N \otimes_{W_\infty} \cO \cong P$. Moreover, Proposition \ref{prop_automorphic_forms_on_D_at_level_Q} and Proposition \ref{prop_automorphic_forms_on_G_at_level_Q} show that the modules $H_{G,N}, H_{D,N}$ are finite free as $\cT \widehat{\otimes}_\cO \cO[\Delta_{Q_N} / (p^N)]$-modules. Finally, completed tensor product with $\cT$ promotes the morphism $P_{Q_N} \to R_{Q_N}$ to a morphism $P_N \to R_N$. 

To patch these objects together we now carry out a diagonalisation argument along very similar lines to the proof of \cite[Proposition 3.3.1]{Kis09}. By \cite[Lemma 2.16]{New19a}, we can find an integer  $g_0  \geq  0$ and for each $N \geq 1$ a surjection $\cO \llbracket X_1, \dots, X_{g_0} \rrbracket \to P_N$ of $\cO$-algebras. Let $\fra \subset W_\infty$ denote the kernel of the augmentation $W_\infty \to \cO$, and for any $N \geq 1$ let $\fra_N$ denote the kernel of the map $W_\infty \to \cO[\Delta_{Q_N}/(p^N)]$. Choose a sequence $(\frb_N)_{N \geq 1}$ of open ideals of $W_\infty$ satisfying the following conditions:
\begin{itemize}
\item For each $N \geq 1$, $\frb_{N+1} \subset \frb_N$.
\item For each $N \geq 1$, $\fra_N \subset \frb_N$.
\item $\cap_{N \geq 1} \frb_N = 0$.
\end{itemize}
Let $s = \max(\dim_E H_D[1/p], \dim_E H_G[1/p])$. Let $r_N = \operatorname{length}_{W_\infty} (W_\infty / \frb_N)^s$. Then the sequence $(r_N)_{N \geq 1}$ is non-decreasing and 
\[ \operatorname{length}_{R_N}  H_{D, N} / (\frb_N) = \operatorname{length}_{W_\infty}  H_{D, N} / (\frb_N) \leq  r_N, \]
so $H_{D, N} / (\frb_N)$ has a natural structure of $R_N / \ffrm_{R_N}^{r_N}$-module. Similarly, $H_{G, N} / (\frb_N)$ has a natural structure of $P_N / \ffrm_{P_N}^{r_N}$-module.
Thus for every pair of integers $N \geq M \geq 1$ we have, by passage to quotient from the data constructed above, a diagram (of rings and modules)
\[ \xymatrix@C=1pc @R=.6pc{ &H_{D, N} / (\frb_M) \ar[r]  &H_{D} / (\frb_M)  \\ R_{loc}\llbracket X_1, \dots, X_g \rrbracket \widehat{\otimes}_\cO W_\infty \ar[r] &\ar@(ul,ur)[] R_N / \ffrm_{R_N}^{r_M} \ar[r] &\ar@(ul,ur)[] R / \ffrm_{R}^{r_M} \\ & & \\ 
\cO \llbracket X_1, \dots, X_{g_0} \rrbracket  \widehat{\otimes}_\cO W_\infty \ar[r] & P_N / \ffrm_{P_N}^{r_M} \ar@(dl,dr)[] \ar[r]  \ar[uu] & P / \ffrm_P^{r_M}\ar@(dl,dr)[]  \ar[uu] \\ & H_{G, N} / (\frb_M)\ar[r]  &H_{G} / (\frb_M)} \]
where the horizontal arrows are all surjective. Keeping $M$ fixed, the cardinalities of the rings and modules appearing in this diagram (excepting those in the first column) are uniformly bounded as $N$ varies. By the pigeonhole principle, we can therefore find an increasing sequence $(N_M)_{M \geq 1}$ of integers $N_M \geq M$ such that for each $M \geq 1$ there is a commutative diagram of $\cO$-algebras
\[  \xymatrix@C=.6pc @R=.6pc{ R_{loc}\llbracket X_1, \dots, X_g \rrbracket \widehat{\otimes}_\cO W_\infty\ar[rr] \ar[dr] & & R_{N_{M+1}} / \ffrm_{R_{N_{M+1}} }^{r_M}  \ar[dl] \ar[rr]  & &  R / \ffrm_{R}^{r_M} \ar[dl] \\
&R_{N_{M}}  / \ffrm_{R_{N_{M}} }^{r_M} \ar[rr] & & R / \ffrm_{R}^{r_M} \\
\cO \llbracket X_1, \dots, X_{g_0} \rrbracket  \widehat{\otimes}_\cO W_\infty \ar[rr] \ar[dr] && P_{N_{M+1}}  / \ffrm_{P_{N_{M+1}} }^{r_M} \ar[uu] \ar[rr] \ar[dl] & &  P / \ffrm_P^{r_M} \ar[uu] \ar[dl] \\
& P_{N_{M}}  / \ffrm_{P_{N_{M}} }^{r_M} \ar[uu] \ar[rr] & &  P / \ffrm_P^{r_M},\ar[uu]} \]
where the morphisms from the back square of the cube to the front are isomorphisms, and there are commutative diagrams of modules
\[ \xymatrix@C=.6pc @R=.6pc{& \ar[dl] H_{D, {N_{M+1}} } / (\frb_M) \ar[rr] & & H_D / (\frb_M) \ar[dl] \\
H_{D, {N_{M}} } / (\frb_M) \ar[rr] & & H_D / (\frb_M) & \\
& \ar[dl] H_{G, {N_{M+1}} } / (\frb_M) \ar[rr] & & H_G / (\frb_M) \ar[dl] \\
H_{G, {N_{M}} } / (\frb_M) \ar[rr] & & H_G / (\frb_M), & } \]
compatible with the module structures arising from the previous commutative cube, and where the arrows from back to front are again isomorphisms. We define 
\[ R_\infty = \varprojlim_M R_{N_M} / \ffrm_{R_{N_M}}^{r_M} \text{ and }P_\infty = \varprojlim_M P_{N_M} / \ffrm_{P_{N_M}}^{r_M}. \]
and similarly
\[ H_{D, \infty} = \varprojlim_M H_{D, N_M} / (\frb_M)  \text{ and }H_{G, \infty} = \varprojlim_M H_{G, N_M} / (\frb_M). \]
By passage to inverse limit, there is a diagram (of rings and modules)
\[ \xymatrix@C=1pc @R=.6pc{ &H_{D, \infty} \ar[r]  &H_{D}   \\ R_{loc}\llbracket X_1, \dots, X_g \rrbracket \widehat{\otimes}_\cO W_\infty \ar[r] &\ar@(ul,ur)[] R_\infty  \ar[r] &\ar@(ul,ur)[] R \\ & & \\ 
\cO \llbracket X_1, \dots, X_{g_0} \rrbracket  \widehat{\otimes}_\cO W_\infty \ar[r] &P_\infty\ar@(dl,dr)[] \ar[r]  \ar[uu] & P \ar@(dl,dr)[]  \ar[uu] \\ & H_{G, \infty} \ar[r]  &H_{G}. } \]
To complete the proof of the proposition, it remains to show that these objects have the following properties:
\begin{itemize}
\item The morphism $R_{loc} \llbracket X_1, \dots, X_g \rrbracket \to R_\infty$ is surjective. 
\item The modules $H_{D, \infty}$, $H_{G, \infty}$ are finite free over $W_\infty$ and the morphism $H_{D, \infty} \to H_D$ (resp. $H_{G, \infty} \to H_G$) factors over an isomorphism $H_{D, \infty} / (\fra) \cong H_D$ (resp. $H_{G, \infty} / (\fra) \cong H_G$).
\item The morphism $R_\infty \to R$ (resp. $P_\infty \to P$) factors over an isomorphism $R_\infty / (\fra) \cong R$ (resp. $P_\infty / (\fra) \cong P$).
\end{itemize}
The first point holds because $R_{loc} \llbracket X_1, \dots, X_g \rrbracket$ is a complete local ring and each map $R_{loc} \llbracket X_1, \dots, X_g \rrbracket \to R_{N_M} / \ffrm_{R_{N_M}}^{r_M}$ is surjective. The second follows from e.g.\ Nakayama's lemma and the freeness of the modules $H_{D, N_M} / \frb_M$, $H_{G, N_M} / \frb_M$. For the third point, we recall that $R_{N_M} / (\fra) = R$, and consequently $R_{N_M} / (\ffrm_{R_{N_M}}^{r_M}, \fra) = R / \ffrm_R^{r_M}$. We therefore need to show that the natural map
\[ \left( \varprojlim_M R_{N_M} / \ffrm_{R_{N_M}}^{r_M} \right) / (\fra) \to \varprojlim_M R_{N_M} / (\ffrm_{R_{N_M}}^{r_M}, \fra) \]
is an isomorphism, or equivalently that the ideal $\fra R_\infty$ of $R_\infty$ is closed in the $\ffrm_{R_\infty}$-adic topology. This is true since $R_\infty$ is a Noetherian ring. The same proof applies to the ring $P_\infty$.
\end{proof}
Now we treat the case $p = 2$.
\begin{proof}[Proof of Theorem \ref{thm_automorphy_of_symmetric_power}, case $p = 2$]
Define $H_G = H(V_0)_{\ffrm_G}$ and $H_D = H_D(U_0)_{\ffrm_D}$. The proof of the theorem in this case will be based on the following proposition, incorporating ideas from the 2-adic  patching argument given in \cite{Kha09a, MR2551765}. If $A \in \cC_\cO$, we follow \cite[\S 2.1]{Kha09a} in writing $\operatorname{Sp}_A$ for the functor $\operatorname{Sp}_A : \cC_\cO \to \operatorname{Sets}$ represented by $A$. We write $\widehat{\bG}_m : \cC_\cO \to \operatorname{Groups}$ for the functor which sends $A \in \cC_\cO$ to the group $\ker(A^\times \to (A / \ffrm_A)^\times)$.
\begin{proposition}\label{prop_construction_of_patched_data_p_equals_2}
We can find an integer $q \geq |S \cup S_\infty| - 2$ with the following property: let $W_\infty = \cO \llbracket Y_1, \dots, Y_q, Z_1, \dots, Z_{4 |S \cup S_\infty| - 1} \rrbracket$ and let $\gamma = 2 - |S \cup S_\infty| + q$. 
Then we can find the following data:
\begin{enumerate}
\item Complete Noetherian local $W_\infty$-algebras $P_\infty$, $R_\infty$ equipped with isomorphisms $P_\infty \otimes_{W_\infty} \cO \cong P$ and $R_\infty \otimes_{W_\infty} \cO \cong R$ in $\cC_\cO$.
\item A complete Noetherian local $\cO$-algebra $R'_\infty$ and surjections $R_{loc} \llbracket X_1, \dots, X_g \rrbracket \to R'_\infty$ and $R'_\infty \to R_\infty$ in $\cC_\cO$, where $g = 2q+1$.
\item A $P_\infty$-module $H_{G,\infty}$ and an $R_\infty$-module $H_{D,\infty}$ such that both $H_{G,\infty}$, $H_{D,\infty}$ are finite free $W_\infty$-modules, together with isomorphisms $H_{G,\infty} \otimes_{W_\infty} \cO \cong H_G$ (as $P$-module) and $H_{D,\infty} \otimes_{W_\infty} \cO \cong H_D$ (as $R$-module).
\item A morphism $P_\infty \to R_\infty$ of $\cO$-algebras making the diagram
\[ \xymatrix{ P_\infty \ar[r] \ar[d] &  R_\infty \ar[d] \\
P \ar[r] & R } \]
commute.
\item A free action of $\widehat{\bG}_m^\gamma$ on $\Sp_{R'_\infty}$ and a $\widehat{\bG}_m^\gamma$-equivariant morphism $\delta : \Sp_{R'_\infty} \to \widehat{\bG}_m^\gamma$, where $\widehat{\bG}_m^\gamma$ acts on itself by the square of the identity.
\end{enumerate}
These objects have the following additional properties:
\begin{enumerate}
\setcounter{enumi}{5}
\item We have $\delta^{-1}(1) = \Sp_{R_\infty}$. The induced action of $\widehat{\bG}_m^\gamma[2](\cO)$ on $R_\infty$ lifts to $H_{D,\infty}$. \label{property_actionlifts} 
\item There exists an action of $\widehat{\bG}_m^\gamma[2]$ on $\Sp_{P_\infty}$ such that the morphism $\Sp_{R_\infty} \to \Sp_{P_\infty}$ is $\widehat{\bG}_m^\gamma[2]$-equivariant,  and the induced action of $\widehat{\bG}_m^\gamma[2](\cO)$ on $P_\infty$ lifts to $H_{G,\infty}$.
\end{enumerate}
\end{proposition}
Once again, we show how Proposition \ref{prop_construction_of_patched_data_p_equals_2} implies the theorem in this case before giving the proof of the proposition. Let $\p \subset P$ denote the kernel of the morphism $P \to \cO$ associated to $t$. It is again enough to show that $\p$ is in the support of $H_G$ as $P$-module, or equivalently that the pullback $\p_\infty \subset P_\infty$ of $\p$ is in the support of $H_{G,\infty}$ as $P_\infty$-module. 

The $P_\infty$-module $H_{G,\infty}$ is a Cohen-Macaulay module, and each irreducible component of $\Supp_{P_\infty} H_{G,\infty}$ has dimension $q + 4 |S \cup S_\infty|$. Similarly, each irreducible component of $\Supp_{R_\infty} H_{D,\infty}$ has dimension $q + 4 |S \cup S_\infty|$. 

Let $R_\infty^{inv} \subset R'_\infty$ denote the subring of invariants for the action of $\widehat{\bG}_m^\gamma$ (cf. \cite[\S 2.4]{Kha09a}). Then the morphism $\Sp_{R'_\infty} \to \Sp_{R^{inv}_\infty}$ is smooth of relative dimension $\gamma$ (one can apply \cite[Proposition 2.5]{Kha09a}, which is used for a very similar purpose in the proof of \cite[Proposition 9.3]{Kha09a}). On the other hand, $\Sp_{R_\infty} \to \Sp_{R_\infty^{inv}}$ is a torsor for the group $\widehat{\bG}_m^\gamma[2]$, showing that $R_\infty^{inv}$ has Krull dimension at least $q + 4 |S \cup S_\infty|$. We conclude that $\Spec R'_\infty$ has dimension at least $q + 4 |S \cup S_\infty| +\gamma = \dim R_{loc} \llbracket X_1, \dots, X_g \rrbracket$. Since $ R_{loc} \llbracket X_1, \dots, X_g \rrbracket$ is a domain, it follows that the map $ R_{loc} \llbracket X_1, \dots, X_g \rrbracket \to R'_\infty$ is an isomorphism, that $\Spec R'_\infty$ is irreducible, that $\Spec R_\infty^{inv}$ is irreducible of dimension $q+4|S\cup S_\infty|$, and that $\widehat{\bG}_m^\gamma[2](\cO)$ acts transitively on the set of irreducible components of $\Spec R_\infty$, all of which have dimension $q+4|S\cup S_\infty|$. Property (\ref{property_actionlifts}) of Proposition \ref{prop_construction_of_patched_data_p_equals_2} implies that $\Supp_{R_\infty} H_{D,\infty}$ is invariant under the action of $\widehat{\bG}_m^\gamma[2](\cO)$, so we conclude that $H_{D,\infty}$ has full support in $\Spec(R_\infty)$ (in fact, considering $H_{D,\infty}[\frac{1}{2}]$ over $R_\infty[\frac{1}{2}]$ we can conclude that $H_{D,\infty}$ is a faithful $R_\infty$-module).

Let $\p' \subset P$ denote the kernel of the morphism $P \to \cO$ associated to $t'$, and let $\p'_\infty$ denote its pullback to $P_\infty$. Then $\p'_\infty \in \Supp_{P_\infty} H_{G,\infty}$, by hypothesis. Similarly, let $\mathfrak{r}, \mathfrak{r}' \subset R$ denote the kernels of the morphisms $R \to \cO$ associated to $r, r'$ respectively, and let $\mathfrak{r}_\infty, \mathfrak{r}'_\infty \subset R_\infty$ denote their pullbacks under the morphism $R_\infty \to R$. Then $\p_\infty$ (resp. $\p'_\infty$) is the image of $\mathfrak{r}_\infty$ (resp $\mathfrak{r}'_\infty$) under the map $\Spec R_\infty \to \Spec P_\infty$. Let $\mathfrak{r}''_\infty \in \Spec R_\infty$ denote a point which is in the $\widehat{\bG}_m^\gamma[2](\cO)$-orbit of $\mathfrak{r}'_\infty$ and on the same irreducible component of $\Spec R_\infty$ as $\mathfrak{r}_\infty$, and let $\p''_\infty$ denote the image of $\mathfrak{r}''_\infty$ in $\Spec P_\infty$. Then $\p''_\infty$, $\p_\infty$ lie on a common irreducible component of $\Spec P_\infty$. Since the action of $\widehat{\bG}_m^\gamma[2](\cO)$ extends to $P_\infty$ and $H_{G,\infty}$ and the morphism $P_\infty \to R_\infty$ is equivariant for this action, $\Supp_{P_\infty} H_{G,\infty}$ is invariant under $\widehat{\bG}_m^\gamma[2](\cO)$ and contains $\p''_\infty$. 

We now observe that the Zariski tangent space of the local ring $P_{\infty, (\p''_\infty)}$ has dimension at most $q + 4|S \cup S_\infty| - 1$. Indeed, translating by the element of $\widehat{\bG}_m^\gamma[2](\cO)$ which takes $\p''_\infty$ to $\p'_\infty$, it suffices to show that the Zariski tangent space of the local ring $P_{\infty, (\p'_\infty)}$ has dimension at most $q +  4 |S \cup S_\infty| - 1$, or  even that $P_{\infty, (\p'_\infty)} / (Y_1,  \dots, Y_q, Z_1, \dots,  Z_{4 |S \cup S_\infty| - 1})  = P_{(\p')}$ is a field. This again follows from \cite[Theorem A, Proposition 2.21, Example 2.34]{New19a}. It follows that $P_{\infty, (\p''_\infty)}$ is a regular local ring of dimension $q +  4 |S \cup S_\infty| - 1$ and that there is a unique irreducible component of $P_\infty$ containing the point $\p''_\infty$, which has dimension $q + 4 |S \cup S_\infty|$ and is contained in $\Supp_{P_\infty} H_{G,\infty}$. We deduce that $\p_\infty \in \Supp_{P_\infty} H_\infty$, as required.

The proof of Proposition \ref{prop_construction_of_patched_data_p_equals_2} is again based on a patching argument. Here is the analogue   of Lemma \ref{lem_existence_of_TW_data_p_neq_2} in our case.
\begin{lemma}
Let $q = \dim_k H^1(F_S / F, \ad \overline{r}) - 2$, and let $g = 2q +1$. Then $q \geq |S \cup S_\infty| - 2$ and for any $N \geq 1$, we can find a Taylor--Wiles  datum $(Q, \widetilde{Q}, (\alpha_v, \beta_v)_{v \in Q}))$ of level $N \geq 1$ with the following properties:
\begin{enumerate}
\item $|Q| = q$.
\item There is a surjection $R_{loc} \llbracket X_1, \dots, X_g \rrbracket \to R_Q^{\prime, \square}$ of $R_{loc}$-algebras.
\item Let $\Theta_Q$ denote the Galois group of the maximal abelian pro-2 extension of $F$ which is unramified outside $Q$ and $(S \cup S_\infty)$-split. Then there is an isomorphism $\Theta_Q  / (2^N) \cong (\bZ / 2^N \bZ)^\gamma$, where $\gamma = 2 - |S \cup S_\infty| + q$.
\end{enumerate}
\end{lemma}
\begin{proof}
This is contained in \cite[Lemma 5.10]{Kha09a}, except that result specifies only that if $v \in Q$ then the eigenvalues $\alpha_v, \beta_v \in k$ of $\overline{r}(\Frob_v)$ are distinct. Here we require that the numbers $\alpha_v^{n-i} \beta_v^{i-1}$ ($i = 1, \dots, n$) are distinct. However, reading the proof of \emph{loc. cit.} we see that we can indeed choose $v$ so that $\overline{r}(\Frob_v)$ satisfies this stronger requirement (using of course our assumption that the projective image of $\overline{r}$ contains $\PSL_2(\bF_{2^a})$ for some $a \geq 1$ such that $2^a > \max(5, 2n-1)$, as we did in the proof of Lemma \ref{lem_existence_of_TW_data_p_neq_2}).
\end{proof}
Now we give the proof of Proposition \ref{prop_construction_of_patched_data_p_equals_2}. Let $q = \dim_k H^1(F_S  / F, \ad \overline{r})-2$ and $g = 2q+1$. For each $N \geq 1$, fix a Taylor--Wiles datum $(Q_N, \widetilde{Q}_N, (\alpha_v, \beta_v)_{v \in Q_N})$ of level $N + \max(m_D, m_G)$ such that there exists a surjection $R_{loc}\llbracket X_1, \dots, X_g \rrbracket \to R_{Q_N}^{\prime,\square}$ of $R_{loc}$-algebras. Define $\cT = \cO \llbracket Z_1, \dots, Z_{4 (|S| + |S_\infty|) - 1} \rrbracket$. Choose for each $N \geq 1$ a surjection $\cO \llbracket Y_1, \dots, Y_q \rrbracket \to \cO[\Delta_{Q_N}]$. Then we get a surjection $W_\infty  \to \cT \widehat{\otimes}_\cO \cO[\Delta_{Q_N}]$.

Define $R_N = R^\square_{Q_N}$, $R'_N = R^{\prime,\square}_{Q_N}$, $H_{D,N} = \cT \widehat{\otimes}_\cO H_D(U_1(Q; N))_{\ffrm_{D,Q_N}}$, $P_N = \cT \widehat{\otimes}_\cO P_{Q_N}$, $H_{G,N} = \cT \widehat{\otimes}_\cO H_G(V_1(Q; N))_{\ffrm_{G,Q_N}}$.  We fix a representative $r^{univ} : G_F \to \GL_2(R)$ for the universal deformation over $R$, and representatives $r_{Q_N}^{univ} : G_F \to \GL_2(R_{Q_N})$ for the universal deformations over $R_{Q_N}$ lifting $r^{univ}$ for each $N \geq 1$. These choices determine compatible isomorphisms $R^\square \cong \cT \widehat{\otimes}_\cO R$ and  $R_{N} \cong \cT \widehat{\otimes}_\cO R_{Q_N}$. Thus each ring $R_N, P_N$ has a $W_\infty$-algebra structure, and there are isomorphisms $R_N \otimes_{W_\infty} \cO \cong R$, $P_N \otimes_{W_\infty} \cO \cong P$. Moreover, Proposition \ref{prop_automorphic_forms_on_D_at_level_Q} and Proposition \ref{prop_automorphic_forms_on_G_at_level_Q} show that the modules $H_{G,N}, H_{D,N}$ are finite free as $\cT \widehat{\otimes}_\cO \cO[\Delta_{Q_N} / (p^N)]$-modules. Completed tensor product with $\cT$ promotes the morphism $P_{Q_N} \to R_{Q_N}$ to a morphism $P_N \to R_N$.

Let $\check{\Theta}_{Q_N} = \Sp_{\cO[\Theta_{Q_N}]}$ denote the group functor $A \mapsto \Hom(\Theta_{Q_N}, A^\times)$. Then $\check{\Theta}_{Q_N}$ acts on $\Sp_{R'_N}$ by twisting: if $r$ is a lifting corresponding to a morphism $R'_N \to A$ and $\chi : \Theta_{Q_N} \to A^\times$ is a character, then $r \otimes \chi$ is a lifting which determines another morphism $R'_N \to A$. There is a morphism $\delta_N : \Sp_{R'_N} \to \check{\Theta}_{Q_N}$ given by taking the determinant and multiplying by $\epsilon$, and $\delta_N^{-1}(1) = \Sp_{R_N}$. The induced action of $\check{\Theta}_{Q_N}[2](\cO)$ on $R_N$ lifts to an action on $H_D(U_1(Q; N))_{\ffrm_{D,Q_N}}$, given by twisting by quadratic characters as above (see also \cite[\S 7.5]{Kha09a}).  Similarly we can define compatible actions of $\check{\Theta}_{Q_N}[2]$ on $P_N$ and of $\check{\Theta}_{Q_N}[2](\cO)$ on $H_G(V_1(Q; N))_{\ffrm_{G,Q_N}}$, which are trivial if $n-1$ is even and which correspond to twisting by the quadratic characters $\chi|_{G_K}$ (resp. $\theta_\chi$ for $\chi \in \check{\Theta}_{Q_N}[2](\cO)$)  when $n-1$ is odd. We extend these to actions on $H_{D, N}$ and $H_{G,N}$ by completed tensor product with $\cT$. The morphism $P_N \to R_N$ is equivariant for these actions. A very similar argument to the proof of \cite[Proposition 9.3]{Kha09a} (with modifications as in the proof of Proposition \ref{prop_construction_of_patched_data_p_odd} above) now shows how to use the above data to construct the objects required by the statement of Proposition \ref{prop_construction_of_patched_data_p_equals_2}. 
\end{proof}
\section{Killing ramification}\label{sec_killing_ramification}

Our goal in this section is to prove the following theorem (Theorem \ref{introthm_sympowers} of the introduction):
\begin{theorem}\label{thm_sym_powers}
Let $n \geq 1$. Let $\pi$ be a regular algebraic, cuspidal automorphic representation of  $\GL_2(\bA_\bQ)$ which is non-CM. Then $\Sym^{n-1} \pi$ exists.
\end{theorem}
We fix $n$, which we can assume to be $\geq 3$. The proof of Theorem \ref{thm_sym_powers} will be roughly by induction on the 
cardinality of $sc(\pi)$, the set of primes $p$ such that $\pi_p$ is 
supercuspidal; the case where $sc(\pi)$ is empty is exactly the main result of 
\cite{New19b}.

We begin with some preparatory definitions and results.
\begin{defn}
Let $\pi$ be a regular algebraic, cuspidal automorphic representation of $\GL_2(\bA_\bQ)$. We define the semisimple conductor $M_\pi$ of $\pi$ to be $M_\pi = \prod_l N((\rec_{\bQ_l} \pi_l)^{ss})$ (where $N$ denotes conductor).
\end{defn}
\begin{lemma}\label{lem_semisimple_conductor}
Let $\pi$ be a regular algebraic, cuspidal automorphic representation of $\GL_2(\bA_\bQ)$, let $p$ be an odd prime, and let $\iota : \overline{\bQ}_p \to \bC$ be an isomorphism. If $\overline{r}_{\pi, \iota}$ is reducible or dihedral\footnote{By dihedral, we mean that the representation is induced from an index two subgroup of $G_{\Q}$.}, then the prime-to-$p$ part of its conductor divides $M_\pi$.
\end{lemma}
\begin{proof}
If $\overline{r}_{\pi, \iota}$ is reducible or dihedral then its image has order prime to $p$,  and for any prime $l \neq p$, $\overline{r}_{\pi, \iota}|_{G_{\bQ_l}}$ is semisimple.
This shows that the conductor of $\overline{r}_{\pi, \iota}|_{G_{\bQ_l}}$ divides the conductor of $(\rec_{\bQ_l} \pi_l)^{ss}$.
\end{proof}
\begin{lemma}\label{lem_replacement_for_breuil_mezard_weak}
Let $\pi$ be a regular algebraic, cuspidal automorphic representation of $\GL_2(\bA_\bQ)$. Let $p$ be a prime, let $\iota : \overline{\bQ}_p \to \bC$ be an isomorphism, and suppose that $\operatorname{Proj} \overline{r}_{\pi, \iota}(G_\bQ)$ contains a conjugate of $\PSL_2(\bF_{p^a})$ for some $p^a > 5$. Then we can find another regular algebraic, cuspidal automorphic representation $\pi'$ of $\GL_2(\bA_\bQ)$ with the following properties:
\begin{enumerate}
\item There is an isomorphism $\overline{r}_{\pi, \iota} \cong \overline{r}_{\pi', \iota}$.
\item $\pi'$ has weight 2 and is not $\iota$-ordinary. 
\item There is an isomorphism $\rec_{\bQ_p} \pi'_p \cong \omega_1 \oplus \omega_2$, where $\omega_1, \omega_2 : W_{\bQ_p} \to \bC^\times$ are characters of conductor dividing $p^3$. 
\item For each prime $l \neq p$, $\pi_l$ is a twist of the Steinberg representation (resp. supercuspidal) if and only if $\pi'_l$ is. 
\end{enumerate}
\end{lemma}
\begin{proof}
Let $\Sigma$ be the set of primes $l \neq p$ such that $\pi_l$ is a twist of the Steinberg representation, and let $T$ be the set of primes $l \neq p$ such that $\pi_l$ is supercuspidal. Let $l_0 \geq 5$ be a prime such that $l_0 \equiv 1 \text{ mod }p$, $\pi_{l_0}$ is unramified, and $\overline{r}_{\pi, \iota}(\Frob_{l_0})$ has distinct eigenvalues (such a prime exists because of our assumption on the image of $\operatorname{Proj} \overline{r}_{\pi, \iota}$). Fix a coefficient field $E$ containing a $p^2$th root of unity. If $l$ is a prime such that $\pi_l$ is supercuspidal, then we can find, after possibly enlarging $E$, an $\cO[\GL_2(\bZ_l)]$-module $M_l$, finite free as $\cO$-module, such that $M_l \otimes_{\cO, \iota} \bC$ is a type for the Bernstein component containing $\pi_l$, in the sense of \cite[Definition A.1.4.1]{breuil-mezard}. We define $M = \otimes_{v \in T} M_l$, $M_k = M \otimes_\cO k$, and $M_E = M \otimes_\cO E$. We write $M^\vee$, $M_k^\vee$ and $M_E^\vee$ for the $\cO$-, $k$-, and $E$-linear duals of these $U$-modules, equipped with the dual action of $U$. 

Let $D$ denote  the quaternion algebra over $\bQ$ such that if $l$ is a prime, then $D$ is ramified at $l$ if and only if $l \in \Sigma$. (Thus $D$ is ramified at $\infty$ if and only if $|\Sigma|$ is odd.) Fix a maximal order $\cO_D \subset D$ and for each prime $l \not\in \Sigma$ an identification $\cO_D \otimes_\bZ \bZ_l \cong M_2(\bZ_l)$. If $U = \prod_l U_l \subset (\cO_D \otimes_\bZ \widehat{\bZ})^\times$ is an open compact subgroup, we write $Y(U)$ for the locally symmetric space of level $U$ (namely the object denoted $X_{\Res_{D / \bQ} \bG_m}^U$ in \cite[\S 3.1]{new-tho}). We regard $M$ as an $\cO[U]$-module by projection to $\prod_{l \in T} \GL_2(\bZ_l)$. If $U_p = \GL_2(\bZ_p)$ and $r \geq 1$ then we write $U_0(p^r) \subset U$ for the open compact subgroup with the same component at primes $l \neq p$ and component 
\[ \left\{ \left( \begin{array}{cc} a & b \\ c &d \end{array}\right) \in \GL_2(\bZ_p) \mid c \equiv 0 \text{ mod }p^r \right\} \]
 at the prime $p$. We identify any pair of characters $\chi_1, \chi_2 : (\bZ / p^r \bZ)^\times \to \cO^\times$ with the character $\chi_1 \otimes \chi_2 : U_0(p^r) \to \cO^\times$ given by the formula 
\[ \chi_1 \otimes \chi_2\left( \left( \begin{array}{cc} a & b \\ c &d \end{array}\right) \right) = \chi_1(a \text{ mod }p^r) \chi_2(d \text{ mod }p^r). \]
We write $M(\chi_1 \otimes \chi_2) = M \otimes_\cO \cO(\chi_1 \otimes \chi_2)$, regarded as $\cO[U_0(p^r)]$-module. We use similar notation for $k$- and $E$-valued characters.

 Let $\delta = 0$ if $D$ is ramified at $\infty$ and $\delta = 1$ otherwise. Let $U = \prod_l U_l$ be the open compact subgroup defined as follows:
 \begin{itemize}
 \item $U_{l_0} = \Iw_{l_0, 1}$.
 \item If $l \not\in \Sigma \cup \{ l_0 \}$, then $U_l = \GL_2(\bZ_l)$.
 \item If $l \in \Sigma$ and $\pi_l \cong \mathrm{St}_2(\chi_l)$, then $U_l = \ker ( \chi_l \circ \det : (\cO_D \otimes_\bZ \bZ_l)^\times \to \bC^\times )$. 
 \end{itemize}
 Then $U$ is neat, in the sense of \cite[\S 3.1]{new-tho} (because of the choice of $U_{l_0}$), and we can find characters $\overline{\chi}_1, \overline{\chi}_2 : \bF_p^\times \to 
k^\times$ such that 
\[ H^\delta(Y(U_0(p)), M_k(\overline{\chi}_1 \otimes \overline{\chi}_2)^\vee)_{\ffrm_\pi} \neq 0, \]
where $\ffrm_\pi \subset \bT^{univ}_{D, \Sigma \cup T \cup \{ l_0, p \}}$ is the maximal ideal associated to $\iota^{-1}\pi^\infty$ (notation for the Hecke algebra as in \S \ref{sec_ALT_for_sym_powers}), cf.~\cite[Corollary 2.12]{bdj}.

Let $\chi_1, \chi_2 : (\bZ / p^3 \bZ)^\times \to \cO^\times$ be lifts of $\overline{\chi}_1, \overline{\chi}_2$ such that $\chi_1 / \chi_2$ has conductor $p^3$. We will show that \begin{equation}\label{eqn_desired_non-vanishing} H^\delta(Y(U_0(p^3)), M_E(\chi_1 \otimes \chi_2)^\vee)^{0 < s < 1}_{\ffrm_\pi} \neq 0,
\end{equation}
where the superscript denotes the subspace where the Hecke operator $[U_0(p^3) \alpha_{p, 1} U_0(p^3)]$ acts with eigenvalues that have $p$-adic valuation $0 < s < 1$. Assuing (\ref{eqn_desired_non-vanishing}) holds, we can complete the proof of the lemma. Indeed, the Jacquet--Langlands correspondence then implies the existence of a regular algebraic, cuspidal automorphic representation $\pi'$ of $\GL_2(\bA_\bQ)$ satisfying the following conditions:
\begin{itemize}
\item There is an isomorphism $\overline{r}_{\pi, \iota} \cong \overline{r}_{\pi', \iota}$.
\item $\pi'$ has weight 2.
\item $\iota^{-1} \pi'_p|_{U_0(p^3)}$ contains a copy of $\chi_1 \otimes 
\chi_2$ on which $[U_0(p^3) \alpha_{p, 1} U_0(p^3)]$ acts with eigenvalue of $p$-adic valuation 
$0 < s <  1$. 
\item If $l \in \Sigma$ then  $\pi'_l$ is a twist of the Steinberg representation.
\item If  $l \in T$ then $\pi'_l|_{\GL_2(\bZ_l)}$ contains $M_l \otimes_{\cO, \iota} \bC$, hence (by definition of a type) $\pi'_l$ is supercuspidal.
\item If $l \not\in \Sigma \cup T \cup \{ p \}$, then $\pi'_l$ is a principal series representation. 
\end{itemize}
By \cite[A.2.4]{breuil-mezard}, there is an isomorphism $\rec_{\bQ_p} \pi'_p 
\cong \omega_1 \oplus \omega_2$, where $\omega_1, \omega_2 : W_{\bQ_p} \to 
\bC^\times$ are characters such that $\omega_1 \circ 
\Art_{\bQ_p}|_{\bZ_p^\times} = \iota \chi_1$ and $\omega_2 \circ 
\Art_{\bQ_p}|_{\bZ_p^\times} = \iota \chi_2$. Moreover, the space 
$\Hom_{U_0(p^3)}(\chi_1 \otimes \chi_2, \iota^{-1} \pi'_p)$ is 1-dimensional. 
Since the Hecke operator $[U_0(p^3) \alpha_{p, 1} U_0(p^3)]$ acts  with eigenvalue of $p$-adic valuation $0 < s <  
1$, $\pi'$ is not $\iota$-ordinary and therefore satisfies our requirements (cf.~\cite[Lemma 5.2]{ger}).

We now show that (\ref{eqn_desired_non-vanishing}) holds. We have
\[ \begin{split} h^\delta(Y(U_0(p^3)), M_E(\chi_1 \otimes \chi_2)^\vee)^{0 < s < 1}_{\ffrm_\pi}  & + h^\delta(Y(U_0(p^3)), M_E(\chi_1 \otimes \chi_2)^\vee)^{s = 0}_{\ffrm_\pi} \\ & + h^\delta(Y(U_0(p^3)), M_E(\chi_1 \otimes \chi_2)^\vee)^{s = 1}_{\ffrm_\pi}  \\ & = h^\delta(Y(U_0(p^3)), M_E(\chi_1 \otimes \chi_2)^\vee)_{\ffrm_\pi}, \end{split} \]
(where lowercase $h$ denotes dimension of cohomology over  $k$  or $E$). For each $\pi'$ contributing to $H^\delta(Y(U_0(p^3)), M_E(\chi_1 \otimes \chi_2)^\vee)$, $\iota^{-1} \pi'_p|_{U_0(p^3)}$ also contains a copy of $\chi_2 \otimes 
\chi_1$ with multiplicity one. The product of the eigenvalues of $[U_0(p^3) \alpha_{p, 1} U_0(p^3)]$ on the $\chi_1\otimes\chi_2$ and $\chi_2\otimes\chi_1$ isotypic spaces of $\iota^{-1} \pi'_p|_{U_0(p^3)}$ has $p$-adic valuation $1$. We deduce that
\[ h^\delta(Y(U_0(p^3)), M_E(\chi_1 \otimes \chi_2)^\vee)^{s = 1}_{\ffrm_\pi} = h^\delta(Y(U_0(p^3)), M_E(\chi_2 \otimes \chi_1)^\vee)^{s = 0}_{\ffrm_\pi}. \]
Moreover,
\[ h^\delta(Y(U_0(p^3)), M_E(\chi_1 \otimes \chi_2)^\vee)^{s = 0}_{\ffrm_\pi}  = h^\delta(Y(U_0(p)),  M_k(\overline{\chi}_1 \otimes  \overline{\chi}_2)^\vee)^{ord}_{\ffrm_\pi} \]
and 
\[ h^\delta(Y(U_0(p^3)), M_E(\chi_2 \otimes \chi_1)^\vee)^{s = 0}_{\ffrm_\pi}  = h^\delta(Y(U_0(p)),  M_k(\overline{\chi}_2 \otimes  \overline{\chi}_1)^\vee)^{ord}_{\ffrm_\pi}, \]
by Hida  theory. It  is  therefore  enough to show that
\begin{multline*} h^\delta(Y(U_0(p^3)), M_E(\chi_1 \otimes \chi_2)^\vee)_{\ffrm_\pi} > h^\delta(Y(U_0(p)),  M_k(\overline{\chi}_1 \otimes  \overline{\chi}_2)^\vee)^{ord}_{\ffrm_\pi}  \\ +  h^\delta(Y(U_0(p)),  M_k(\overline{\chi}_2 \otimes  \overline{\chi}_1)^\vee)^{ord}_{\ffrm_\pi}, 
\end{multline*}
or even  that
\begin{multline*} h^\delta(Y(U_0(p^3)), M_k(\overline{\chi}_1 \otimes \overline{\chi}_2)^\vee)_{\ffrm_\pi} > h^\delta(Y(U_0(p)),  M_k(\overline{\chi}_1 \otimes  \overline{\chi}_2)^\vee)_{\ffrm_\pi} \\ +  h^\delta(Y(U_0(p)),  M_k(\overline{\chi}_2 \otimes  \overline{\chi}_1)^\vee)_{\ffrm_\pi}. 
\end{multline*}
Using the exactness of $H^\delta(Y(U_0(p)), (?)^\vee)_{\ffrm_\pi}$ as a (contravariant) functor of smooth $k[U_0(p)]$-modules, it is therefore enough to show that 
$\Ind^{U_0(p)}_{U_0(p^3)} \overline{\chi}_1 \otimes \overline{\chi}_2$ contains 
$\overline{\chi}_1 \otimes \overline{\chi}_2$ and $\overline{\chi}_2 \otimes 
\overline{\chi}_1$ as Jordan--H\"older factors with multiplicity at least 2 (or 
when $\overline{\chi}_1 = \overline{\chi}_2$, that it contains 
$\overline{\chi}_1 \otimes \overline{\chi}_2$ with multiplicity at least 3). 

This is true. Indeed, the 
semisimplification of a smooth $k[\Iw_p]$-module (say finite-dimensional as $k$-vector space) is determined by its 
restriction to the diagonal torus in $\Iw_p$, and we can then use 
Mackey's formula to show that if $\overline{\psi} : \bF_p^\times \to k^\times$ 
is a character then the semisimplification of $\Ind^{U_0(p)}_{U_0(p^3)} 
\overline{\chi}_1 \otimes \overline{\chi}_2$ contains 
$\overline{\chi}_1\overline{\psi} \otimes 
\overline{\chi}_2\overline{\psi}^{-1}$ with multiplicity $p+2$ if 
$\overline{\psi} = 1$ and multiplicity $p+1$ if $\overline{\psi} \neq 1$. 
\end{proof}
\begin{lemma}\label{lem_replacement_for_breuil_mezard}
Let $\pi$ be a regular algebraic, cuspidal automorphic representation of $\GL_2(\bA_\bQ)$. Let $p \geq 3$ be a prime, let $\iota : \overline{\bQ}_p \to \bC$ be an isomorphism, and suppose that $\operatorname{Proj} \overline{r}_{\pi, \iota}(G_\bQ)$ contains a conjugate of $\PSL_2(\bF_{p^a})$ for some $p^a > 5$. Then we can find another regular algebraic, cuspidal automorphic representation $\pi'$ of $\GL_2(\bA_\bQ)$ with the following properties:
\begin{enumerate}
\item There is an isomorphism $\overline{r}_{\pi, \iota} \cong \overline{r}_{\pi', \iota}$.
\item $\pi'$ has weight 2 and is not $\iota$-ordinary. 
\item There is an isomorphism $\rec_{\bQ_p} \pi'_p \cong \omega_1 \oplus \omega_2$, where $\omega_1, \omega_2 : W_{\bQ_p} \to \bC^\times$ are characters of conductor dividing $p^3$. In particular, $N(\pi'_p) | p^6$.
\item For each prime $l \neq p$, $r_{\pi, \iota}|_{G_{\bQ_l}} \sim r_{\pi', \iota}|_{G_{\bQ_l}}$. In particular, $N(\pi_l) = N(\pi'_l)$. 
\end{enumerate}
\end{lemma}
\begin{proof}
Lemma \ref{lem_replacement_for_breuil_mezard_weak} implies the existence of a $\pi'$ satisfying requirements (1)--(3). We can appeal to \cite[Corollary 3.1.7]{gee061} to replace it with a $\pi'$ also satisfying $r_{\pi, \iota}|_{G_{\bQ_l}} \sim r_{\pi', \iota}|_{G_{\bQ_l}}$ for each prime $l \neq p$. By purity, $r_{\pi, \iota}|_{G_{\bQ_l}}$ and $r_{\pi', \iota}|_{G_{\bQ_l}}$ `strongly connect' to each other in the terminology of \cite[\S 1.3]{BLGGT}. Remark (6) of \cite[p.\ 524]{BLGGT} implies that $N(\pi_l) = N(\pi'_l)$.
\end{proof}
\begin{defn}\label{def_seasoned}
Let $\pi$ be a regular algebraic, cuspidal automorphic representation of $\GL_2(\bA_\bQ)$ of weight 2 such that $2 \not\in sc(\pi)$ and let $q, t, r$ be prime numbers. We say that $\pi$ is seasoned with respect to $(q, t, r)$ if the following properties hold:
\begin{enumerate}
\item $t$ divides $q+1$, $t \not\in sc(\pi)$, and $t > \max(10,  8n(n-1))$ and $(q+1) / t > 2$.
\item There is an isomorphism $\rec_{\bQ_q} \pi_q \cong \Ind_{W_{\bQ_{q^2}}}^{W_{\bQ_{q}}} \chi$, where $\chi : W_{\bQ_{q^2}} \to \bC^\times$ is a character such that $\chi|_{I_{\bQ_q}}$ has order $t$. (In particular, $q \in sc(\pi)$ and $N(\pi_q) = q^2$.)
\item $r$ is a primitive root modulo $q$. If $M$ denotes the least common multiple of the prime-to-$q$ part of $M_\pi$ and $\prod_{p \in sc(\pi) - \{ q \}} p^{6}$, then $r \equiv 1 \text{ mod }M$.
\item $\pi_r$ is an unramified twist of the Steinberg representation.
\item For each prime $p  \in sc(\pi)$ and for each irreducible dihedral representation $\overline{\rho} : \Gal(\overline{\bQ}/\bQ) \to \GL_2(\overline{\bF}_p)$ of prime-to-$p$ conductor dividing $M q^2$, there exists a prime number $s$ such that $\pi_{s}$ is an unramified twist of the Steinberg representation, $\overline{\rho}(\Frob_s)$ is scalar, and $s \not\equiv 1 \text{ mod }p$.
\end{enumerate}
\end{defn}
\begin{proposition}\label{prop_seasoned_large_image}
If $\pi$ is seasoned with respect to $(q, t, r)$ then for each prime $p \in sc(\pi)$ there exists an isomorphism $\iota : \overline{\bQ}_p  \to \bC$ such that  $\overline{r}_{\pi, \iota}(G_\bQ)$  contains a conjugate of $\SL_2(\bF_{p^a})$ for some $p^a > 2n-1$.
\end{proposition}
\begin{proof}
We split into cases depending on whether or not $p = q$. First suppose that $p 
= q$, and fix an isomorphism $\iota : \overline{\bQ}_q \to \bC$. We first claim 
that $\overline{r}_{\pi, \iota}$ is irreducible. Otherwise, there's an 
isomorphism $\overline{r}_{\pi, \iota} \cong \overline{\chi}_1 \oplus 
\overline{\chi}_2$ for characters $\overline{\chi}_i : G_\bQ \to 
\overline{\bF}_q^\times$ of prime-to-$q$ conductor dividing $M_\pi$. 

Let $\omega_2 : I_{\bQ_q} \to \overline{\bQ}_q^\times$ be the Teichm\"{u}ller lift of the fundamental character of niveau 2. Let $b = (q+1) / t$. Then $\iota^{-1} \chi|_{I_{\bQ_q}} = \omega_2^{a(q-1)b}$ for some integer $a \in \{1, \dots, t-1 \}$ (this tame character has niveau $2$ since it extends to $W_{\Q_{q^2}}$ and its order $t$ does not divide $q-1$). Write $a(q-1)b = i + (q+1)j$ for some $i \in \{ 1, \dots, q \}$. Then (\cite[Theorem 4.6.1]{gee061}) we have $(\overline{\chi}_1 / \overline{\chi}_2)|_{I_{\bQ_q}} = \epsilon^{i-1}$ or $\epsilon^{1-i}$. After relabelling, we can assume that $(\overline{\chi}_1 / \overline{\chi}_2)|_{I_{\bQ_q}} = \epsilon^{1-i}$.

Since $r \equiv 1 \text{ mod }M$ (in particular, the characters $\overline{\chi}_i$ are unramified at $r$), we have $(\overline{\chi}_1 / 
\overline{\chi}_2)(\Frob_r) = r^{i-1}$. Since $\pi_r$ is an unramified twist of 
the Steinberg representation, we have $(\overline{\chi}_1 / 
\overline{\chi}_2)(\Frob_r) = r$ or $r^{-1}$. Since $r$ is a primitive root 
modulo $q$, this implies that one of $i$, $i-2$ is divisible by $q-1$.

Since $b$ divides $i$, $i$ is among the numbers $b, 2b, \dots, q+1-b$. Since $b > 2$, we see that neither $i$ nor $i-2$ can be divisible by $q-1$. This contradiction implies that  $\overline{r}_{\pi,  \iota}$ is irreducible.

We next claim that $\overline{r}_{\pi, \iota}$ is not dihedral. If 
$\overline{r}_{\pi, \iota}$ is dihedral, then Lemma \ref{lem_semisimple_conductor} shows that the prime-to-$q$ part of the conductor of 
$\overline{r}_{\pi, \iota}$ divides $M_\pi$, so there exists a prime number $s$ 
such that  $\pi_s$  is an unramified twist of the Steinberg representation, 
$\overline{r}_{\pi, \iota}(\Frob_s)$ is scalar, and $s \not\equiv 1 \text{ mod 
}p$. This is a contradiction.

To finish the proof in the case $p=q$, we need to make a particular choice of 
$\iota$. Such a choice fixes the value of $a \in \{ 1, \dots, t-1 
\}$; conversely,  any $a  \in \{ 1, \dots, t-1 
\}$ can be obtained by making a suitable choice of $\iota$.
We  choose $a$ so that $i = b$. Invoking \cite[Theorem 4.6.1]{gee061} once 
more, we see that the projective image of $\overline{r}_{\pi, \iota}$ contains 
an element of order in  the set \[\{  (q+1)  / \gcd(q+1, i+1), (q+1) / \gcd(q+1, i-1), 
(q-1)/\gcd(q-1, i-1) \},\] therefore of order at least $t/2$.  Since $t /2 >5$, the 
classification  of  finite subgroups of $\PGL_2(\overline{\bF}_q)$ shows that 
the projective image  of   $\overline{r}_{\pi, \iota}$ contains  
$\PSL_2(\bF_{q^a})$ (and is contained  in $\PGL_2(\bF_{q^a})$) for some $a \geq 
1$. If $q^a \leq 2n-1$ then $q^{2a} -1 \leq 4n(n-1)$, so every element of 
$\PGL_2(\bF_{q^a})$ of order  prime to $q$ has order at most $4n(n-1)$.  Since  
$t / 2 > 4n(n-1)$, we see that  we  must  have $q^a > 2n-1$.

Now suppose that  $p \neq  q$, and fix  an isomorphism  $\iota : \overline{\bQ}_p \to \bC$. Then $p \neq t$ and $t \nmid q-1$, so $\overline{r}_{\pi, \iota}|_{G_{\bQ_q}}$ is irreducible and its projective image contains elements  of order  $t$, and so $\overline{r}_{\pi, \iota}$ is irreducible and its projective image contains  elements of order  $t$. Using again the classification of finite subgroups of $\PGL_2(\overline{\bF}_p)$, we see that to complete the proof we just need to show that $\overline{r}_{\pi, \iota}$ is not dihedral. If it is dihedral then there exists a prime number $s$ such that $\pi_s$ is an unramified twist of the Steinberg representation, $\overline{r}_{\pi, \iota}(\Frob_s)$ is scalar,  and $s  \not\equiv 1 \text{ mod }p$. This is a contradiction.
\end{proof}
To prove the next proposition, we need to find primes with special properties, namely  that their Frobenius elements act on the composita of certain field extensions in a prescribed way. Using the Chebotarev density theorem, we see that it is equivalent to exhibit Galois automorphisms acting in the correct way. In order to do so,  it is helpful to recall the following lemma from basic Galois theory.
\begin{lemma}\label{lem_Galois_group_of_compositum}
Let $E / K$ be a finite Galois extension, and let $K_1 / K, K_2 / K$ be Galois subextensions. Then the natural map $\Gal(K_1 K_2 / K) \to \Gal(K_1 / K) \times_{\Gal(K_1 \cap K_2 / K)} \Gal(K_2 / K)$ is an isomorphism.
\end{lemma}
\begin{proposition}\label{prop_congruence_to_seasoned}
Let $\pi$ be a regular algebraic, cuspidal automorphic representation of $\GL_2(\bA_\bQ)$ of weight 2 which is non-CM. Suppose  that $2, 3 \not\in sc(\pi)$. Then we  can find a  regular algebraic, cuspidal automorphic representation $\pi'$  of $\GL_2(\bA_\bQ)$ with the following properties:
\begin{enumerate}
\item There exist prime  numbers $(q, t, r)$ such that $\pi'$ is seasoned with respect to $(q, t, r)$ and $sc(\pi') = sc(\pi) \cup  \{ q \}$.
\item $\Sym^{n-1} \pi'$ exists if and only if $\Sym^{n-1} \pi$ does. 
\end{enumerate}
\end{proposition}
\begin{proof}
Fix a prime $t > \max(10,8n(n-1), N(\pi))$  such that $t \equiv 1 \text{ mod }4$  and there exists an isomorphism $\iota : \overline{\bQ}_t \to \bC$ such that  $G = \operatorname{Proj} \overline{r}_{\pi, \iota}(G_\bQ)$ is conjugate either to $\PSL_2(\bF_t)$ or $\PGL_2(\bF_t)$. Since $t > 5$, the group $\PSL_2(\bF_t)$ is simple. The condition $t \equiv 1 \text{ mod }4$ implies that $-1 \text{ mod }t$ is a square and that the image of complex conjugation $c$ in $G$ lies in $[G, G]$. Using the Chebotarev density theorem, we can therefore choose a prime $q$  such that $q  \equiv -1 \text{ mod }t$, $(q+1) > 2t$, and the image of $\Frob_q$ in $G$ is in the  conjugacy class of complex conjugation.

Similarly, we can choose a prime $r$ such that $r$ is a primitive root modulo 
$q$, $\overline{r}_{\pi, \iota}(\Frob_r)$ is scalar, and $r$ splits in 
$\bQ(\zeta_{M_\pi t})$ and in $\bQ(\zeta_{p^{6}})$ for every $p \in sc(\pi)$.  Indeed, the prime $q$ is unramified in $\overline{\bQ}^{\ker \operatorname{Proj}\rbar_{\pi,\iota}}(\zeta_{M_\pi t}, \{ \zeta_{p^{6}} \}_{p \in sc(\pi)})$ but totally ramified in $\bQ(\zeta_q)$, so the intersection of these two fields in $\overline{\bQ}$ is $\bQ$. We choose $r$ so that it splits in the first field and is totally inert in the second.

Let $sc(\pi) = \{ p_1, \dots, p_k \}$ and let $p_{k+1} = q$. For each $i = 1, 
\dots, k+1$, let $\overline{\rho}_{i, j}$ ($j \in X_i$) be a set of 
representatives for the (finitely many) conjugacy classes of irreducible 
dihedral representations $\overline{\rho} : G_\bQ \to 
\GL_2(\overline{\bF}_{p_i})$ of prime-to-$p_i$ conductor dividing $\operatorname{lcm}(M_\pi q^2, \prod_{p \in sc(\pi)} p^6)$. For any 
$(i, j)$, the abelianization of the projective image of $\overline{\rho}_{i, 
j}$ is isomorphic either to $\bZ / 2 \bZ$ or $(\bZ / 2 \bZ)^2$. In either case 
 we claim that we can find a prime $s_{i, j}$ such that 
$\overline{\rho}_{i, j}(\Frob_{s_{i, j}})$ is scalar, $s_{i, j} \not\equiv 1 
\text{ mod }p_i$, the image of $\Frob_{s_{i, j}}$ in $G$ is in the conjugacy 
class of complex conjugation, and $s_{i, j} \equiv -1 \text{ mod }t$. 

To see this, let $E_1 =  \overline{\bQ}^{\ker  \operatorname{Proj} \overline{\rho}_{i,  j}}$ and $E_2 = \overline{\bQ}^{\ker \operatorname{Proj} \overline{r}_{\pi, \iota}}$.  We want to show there is $\sigma \in \Gal(E_1 E_2(\zeta_{p_i}, \zeta_t) / \bQ)$ such that $\sigma|_{E_1} = 1$,  $\sigma|_{\bQ(\zeta_{p_i})} \neq 1$ and $\sigma|_{E_2(\zeta_t)} = c$. First we find $\tau \in \Gal(E_1 (\zeta_{p_i}) / \bQ)$ such that $\tau|_{E_1} = 1$ and $\tau|_{\bQ(\zeta_{p_i})} \neq 1$. $\Gal(E_1 / \bQ)$ is soluble and its maximal abelian quotient is a quotient of $(\bZ / 2 \bZ)^2$, so $\Gal(E_1 \cap \bQ(\zeta_{p_i}) / \bQ)$ is a quotient of $(\bZ / 2 \bZ)^2$. This shows that $E_1 \cap \bQ(\zeta_{p_i})$ is either trivial or quadratic. Since $(\bF_{p_i}^\times)^2$  contains 
non-identity elements (because $p_i \geq 5$, because $p_i \in sc(\pi)$) we can find a $\tau$ with the desired property using Lemma \ref{lem_Galois_group_of_compositum}. We can assume that $\tau$ acts trivially on the maximal abelian subextension of $E_1(\zeta_{p_i})$ of exponent 2.

Using Lemma \ref{lem_Galois_group_of_compositum} again, we're done if we can show that $\tau|_{E_1(\zeta_{p_i}) \cap E_2(\zeta_t)} = c|_{E_1(\zeta_{p_i}) \cap E_2(\zeta_t)}$. Let $E_2^{ab}$ denote the maximal abelian subfield of $E_2$. It has degree 1 or 2 over $\bQ$ (because of the form of the image of $\overline{r}_{\pi, \iota}$) and $\Gal(E_2 / E_2^{ab})$ is a non-abelian simple group. Thus the maximal soluble quotient of $\Gal(E_2(\zeta_t) / \bQ)$ is $\Gal(E_2^{ab}(\zeta_t) / \bQ)$, which is in fact abelian. Since $\Gal(E_1(\zeta_{p_i}) / \bQ)$ is soluble, this shows that $\Gal(E_1(\zeta_{p_i}) \cap E_2(\zeta_t) / \bQ)$ is abelian. Since $t$ is coprime to $qN(\pi)$, the prime $t$ is unramified in $E_1(\zeta_{p_i})$, while the quotient of $[E^{ab}_2(\zeta_t) : \bQ]$ by the ramification index of $t$ is 1 or 2. We conclude that $E_1(\zeta_{p_i}) \cap E_2(\zeta_t)$ is either trivial or quadratic. In particular, $\tau$ acts trivially on it. The element $c$ also acts trivially on it, since it acts trivially on $E_2^{ab}$ and also on the quadratic subfield of $\bQ(\zeta_t)$ (since $t \equiv 1 \text{ mod }4$). This completes the proof of the claim.

To conclude the proof of the proposition, we apply \cite[Corollary 3.1.7]{gee061}; it implies the 
existence of a regular algebraic, cuspidal automorphic representation $\pi'$ of 
$\GL_2(\bA_\bQ)$  of weight 2 such that $\overline{r}_{\pi', \iota} \cong 
\overline{r}_{\pi, \iota}$, such that  for each $s \in \{ r, s_{i, j} \}$, 
$\pi'_s$ is an unramified twist of the Steinberg representation, such that 
$\rec_{\bQ_q}  \pi'_q \cong \Ind_{W_{\bQ_{q^2}}}^{W_{\bQ_q}} \chi$   for a 
character $\chi : W_{\bQ_{q^2}} \to \bC^\times$  such that $\chi|_{I_{\bQ_q}}$ 
has order $t$, and such that for every other prime $p$,  we have $r_{\pi,  
\iota}|_{G_{\bQ_p}} \sim r_{\pi', \iota}|_{G_{\bQ_p}}$, with notation  as in 
\cite[\S 1]{BLGGT}. (The hypothesis `$(ord)$' of \cite[Proposition 3.1.5]{gee061} is automatic in our situation.)  In particular, we have $M_{\pi'} = q^2 M_\pi$. We see 
that $\pi'$ is seasoned with respect to $(q, t, r)$. To see that $\Sym^{n-1}  
\pi$ exists if and only if $\Sym^{n-1} \pi'$ does, apply e.g.\ \cite[Theorem 
4.2.1]{BLGGT}. The potential diagonalizability assumption is satisfied because $r_{\pi, \iota}|_{G_{\bQ_t}}, r_{\pi', \iota}|_{G_{\bQ_t}}$ are both Fontaine--Laffaille, while the representations $\Sym^{n-1} \overline{r}_{\pi, \iota} \cong \Sym^{n-1} \overline{r}_{\pi', \iota}$ are irreducible because the $m^\text{th}$ symmetric power of the standard representation of $\SL_2(\bF_t)$ is irreducible whenever $t > m$. 
\end{proof}
\begin{proposition}\label{prop_seasoned_reduce_the_level}
Let $\pi$ be a regular algebraic, cuspidal automorphic representation of  $\GL_2(\bA_\bQ)$ of weight $2$. Suppose that $\pi$ is seasoned with respect to $(q, t, r)$, and let $p \in sc(\pi)$ satisfy $p \geq 5$. Then we can find a regular algebraic, cuspidal automorphic representation $\pi'$ of $\GL_2(\bA_\bQ)$ with the following properties:
\begin{enumerate}
\item $\pi'$ has weight 2 and is non-CM.
\item $sc(\pi') = sc(\pi) - \{ p \}$. If $p \neq q$, then $\pi'$ is seasoned with respect to $(q, t, r)$.
\item If $\Sym^{n-1} \pi'$ exists, then so does $\Sym^{n-1} \pi$.
\end{enumerate}
\end{proposition}
\begin{proof}
By Theorem  \ref{thm_automorphy_of_symmetric_power}, it's enough to find an  isomorphism $\iota : \overline{\bQ}_p \to \bC$ and a regular algebraic, cuspidal automorphic representation $\pi'$ of $\GL_2(\bA_\bQ)$ with the following  properties:
\begin{itemize}
\item The image of $\overline{r}_{\pi, \iota}$ contains a conjugate of $\SL_2(\bF_{p^a})$ for some $p^a > 2n-1$.
\item $\pi'$ has weight  2  and is non-CM. 
\item $\overline{r}_{\pi', \iota} \cong \overline{r}_{\pi, \iota}$.
\item $sc(\pi') = sc(\pi) - \{ p \}$.
\item For  any prime $l \neq p$, $r_{\pi', \iota}|_{G_{\bQ_l}} \sim r_{\pi, \iota}|_{G_{\bQ_l}}$.
\item $\pi'$ is not $\iota$-ordinary.
\item If $p \neq q$ then the conductor of $\pi'_p$ divides $p^{6}$.
\end{itemize}
If $p \neq q$, then the last condition ensures that $\pi'$ is still seasoned 
with respect to $(q, t, r)$ (more precisely, that conditions (3) and (5) in Definition \ref{def_seasoned} still hold). We choose $\iota$ satisfying the first condition using Proposition \ref{prop_seasoned_large_image}; then the existence of a $\pi'$ satisfying the above requirements is the content of Lemma \ref{lem_replacement_for_breuil_mezard}.
\end{proof}
\begin{proposition}\label{prop_sym_powers_weight_2_2_3}
Let $\pi$ be a regular algebraic, cuspidal automorphic representation of $\GL_2(\bA_\bQ)$. Suppose that $\pi$ is of weight $2$ and non-CM, and suppose that $2, 3 \not\in sc(\pi)$. Then $\Sym^{n-1} \pi$ exists.
\end{proposition}
\begin{proof}
If $k \geq 0$,  let $(H_k)$ denote the hypothesis that the conclusion of the  proposition  holds when $|sc(\pi)| \leq k$, and   let  $(H'_k)$ denote the hypothesis that the conclusion of the proposition holds when  $|sc(\pi)| \leq k$ and $\pi$ is seasoned with respect to some tuple $(q, t, r)$. As remarked above, $(H_0)$ follows from the results of  \cite{New19b}. It therefore suffices to prove the implications $(H_k) \Rightarrow (H'_{k+1})$ and $(H'_k) \Rightarrow (H_k)$.

The  first implication follows immediately from Proposition \ref{prop_seasoned_reduce_the_level}. For the second, assume that $(H'_k)$ holds and let $\pi$ be a regular algebraic, cuspidal automorphic representation of $\GL_2(\bA_\bQ)$ which is of weight 2 and non-CM, and such that $|sc(\pi)| = k \geq 1$. By Proposition \ref{prop_congruence_to_seasoned}, we can find a regular algebraic, cuspidal automorphic representation $\pi'$ of $\GL_2(\bA_\bQ)$ which is seasoned with respect to $(q, t, r)$, such that $sc(\pi') = sc(\pi) \cup \{ q \}$,  and such that the existence of $\Sym^{n-1} \pi$ is equivalent to the existence of $\Sym^{n-1} \pi'$.

Now choose a prime $p \in sc(\pi)$ (so $p \in sc(\pi')$ and $p \neq q$). Applying Proposition \ref{prop_seasoned_reduce_the_level} with this choice of $p$ gives another regular algebraic, cuspidal automorphic representation $\pi''$ of $\GL_2(\bA_\bQ)$ which is seasoned with respect to $(q, t, r)$, such that $|sc(\pi'')| = k$, and such that the existence of $\Sym^{n-1} \pi''$ implies that of $\Sym^{n-1} \pi'$. The existence of $\Sym^{n-1} \pi''$ follows from $(H'_k)$, so we're done.
\end{proof}
We can now give the proof of Theorem \ref{thm_sym_powers}.
\begin{proof}[Proof of  Theorem \ref{thm_sym_powers}]
Let $\pi$ be a non-CM, regular algebraic, cuspidal automorphic representation of $\GL_2(\bA_\bQ)$. We must show that $\Sym^{n-1} \pi$ exists. We first do this under the additional assumption that $\pi$ has weight 2 and that $2 \not\in sc(\pi)$. We can assume that $\pi_3$ is supercuspidal. Fix a prime $t > \max(5, 4n(n-1), N(\pi))$ such that $t \equiv 1 \text{ mod }4$ and there exists an isomorphism $\iota_t : \overline{\bQ}_t \to \bC$ such that $G = \operatorname{Proj} \overline{r}_{\pi, \iota_t}(G_\bQ)$ is conjugate either to $\PSL_2(\bF_t)$ or $\PGL_2(\bF_t)$. Using the Chebotarev density theorem, we can find a prime $q$ satisfying the following conditions:
\begin{itemize}
\item The prime $q$  satisfies $q \equiv -1 \text{ mod }t$, $q \equiv 1 \text{ mod }8$, and $q \equiv 1 \text{ mod }l$ for every prime $l < t$. 
\item The image of $\Frob_q$ in $G$ is in the conjugacy class of complex conjugation.
\end{itemize}
(Compare \cite[Lemma 8.2]{Kha09}.) By \cite[Corollary 3.1.7]{gee061}, we can find another regular algebraic, cuspidal automorphic representation $\pi'$ of weight 2 satisfying the following conditions:
\begin{itemize}
\item  $\overline{r}_{\pi, \iota_t} \cong \overline{r}_{\pi', \iota_t}$.
\item If $l \neq q$ is a prime, then $r_{\pi, \iota_t}|_{G_{\bQ_l}} \sim r_{\pi', \iota_t}|_{G_{\bQ_l}}$.
\item There is an isomorphism $\rec_{\bQ_q} \pi'_q \cong \Ind_{W_{\bQ_q^2}}^{W_{\bQ_q}} \chi$, where $\chi : W_{\bQ_{q^2}} \to \bC^\times$ is a character such that $\chi|_{I_{\bQ_q}}$ has order $t$.
\end{itemize}
Applying \cite[Theorem 4.2.1]{BLGGT} to $\Sym^{n-1} r_{\pi, \iota_t}$, we see that $\Sym^{n-1} \pi$ exists if and only if $\Sym^{n-1} \pi'$ does. (The potential diagonalizability assumption is satisfied because $r_{\pi, \iota_t}|_{G_{\bQ_t}}, r_{\pi', \iota_t}|_{G_{\bQ_t}}$ are both Fontaine--Laffaille, while the representations $\Sym^{n-1} \overline{r}_{\pi, \iota_t} \cong \Sym^{n-1} \overline{r}_{\pi', \iota_t}$ are irreducible because the $m^\text{th}$ symmetric power of the standard representation of $\SL_2(\bF_t)$ is irreducible whenever $t > m$. )

Let $\iota : \overline{\bQ}_3 \to \bC$ be an isomorphism. Then (\cite[Lemma 6.3]{Kha09}) there exists $a \geq 2$ such that the image of $\operatorname{Proj}\overline{r}_{\pi, \iota}$ is conjugate to $\PSL_2(\bF_{3^a})$ or $\PGL_2(\bF_{3^a})$. In fact, we must have $3^a > 2n-1$: otherwise $t \leq 3^{2a} - 1 \leq 4n(n-1)$, a contradiction to our assumption $t > 4n(n-1)$.

Applying Lemma \ref{lem_replacement_for_breuil_mezard_weak}, we can find another regular algebraic, cuspidal automorphic representation  $\pi''$ of $\GL_2(\bA_\bQ)$ of weight 2 such  that $\overline{r}_{\pi', \iota} \cong \overline{r}_{\pi'', \iota}$,  $2, 3 \not\in sc(\pi'')$,  $r_{\pi'', \iota}|_{G_{\bQ_3}}$ is potentially Barsotti--Tate and non-ordinary, and for each prime $l \neq 3$, $\pi'_l$ is a twist of the Steinberg representation if and only if $\pi''_l$ is. Then Proposition \ref{prop_sym_powers_weight_2_2_3} implies that $\Sym^{n-1} \pi''$ exists. We can then invoke Theorem \ref{thm_automorphy_of_symmetric_power} to see that $\Sym^{n-1} \pi'$ exists and hence $\Sym^{n-1} \pi$ exists.

The next case to treat is when $\pi$ has weight 2 but now $2 \in sc(\pi)$. In this case we can repeat the same argument with $3$ replaced by $2$ to conclude the existence of $\Sym^{n-1} \pi$.

Finally we treat the general case where $\pi$  has weight $k$ for some $k > 2$ and we make no assumption on $sc(\pi)$. In this case  we can 
find a prime  $t > \max(5, k(n+1))$ such that $\pi_t$ is unramified and an 
isomorphism $\iota  : \overline{\bQ}_t \to \bC$ such that the image of 
$\overline{r}_{\pi, \iota}$ contains a conjugate of $\SL_2(\bF_t)$. Applying Lemma \ref{lem_replacement_for_breuil_mezard_weak} again, we can find another 
regular algebraic, cuspidal automorphic representation $\pi'$ of 
$\GL_2(\bA_\bQ)$ of weight 2 such that $\overline{r}_{\pi', \iota} \cong 
\overline{r}_{\pi, \iota}$ and $r_{\pi', \iota}$ is potentially Barsotti--Tate. 
Then $\Sym^{n-1} \pi'$ exists and $\Sym^{n-1} \overline{r}_{\pi, \iota}$ is 
irreducible. We can now apply 
\cite[Theorem 4.2.1]{BLGGT} to conclude that $\Sym^{n-1} r_{\pi, \iota}$ is 
automorphic and therefore that $\Sym^{n-1} \pi$ exists. Note that $\Sym^{n-1} {r}_{\pi', \iota}$ is the symmetric power of a $2$-dimensional potentially 
diagonalizable representation (\cite[Lemma 4.4.1]{geekisin}), and hence potentially diagonalizable, cf.~the remark following \cite[Definition 3.3.5]{blgg}.
\end{proof}
\appendix
\section{The case of weight one forms}\label{wt1}
In this short appendix we record the automorphy of the symmetric power lifting for cuspidal Hecke eigenforms of weight $1$, or with CM. The most difficult case is due to Kim \cite[Theorem 6.4]{Kim2004}.
\begin{theorem}
Let $n \ge 1$. Let $\pi$ be a cuspidal automorphic representation of $\GL_2(\A_\Q)$ with $\pi_{\infty}$ holomorphic limit of discrete series, or with $\pi$ the automorphic induction of a Hecke character for a quadratic field. Then $\Sym^{n}\pi$ exists.
\end{theorem}
Note that in these cases $\Sym^{n}\pi$ is usually not cuspidal.
\begin{proof}
First we assume that $\pi_\infty$ is holomorphic limit of discrete series. Twisting by an algebraic Hecke character, we can assume that $\pi$ is generated by a holomorphic weight $1$ cuspidal Hecke eigenform. In particular, Deligne and Serre \cite{deligne-serre} constructed a continuous odd irreducible representation 
$r_\pi: G_\Q\to \GL_2(\C)$ with $r_\pi|_{W_{\Qp}} \cong \rec^T_{\Qp}(\pi_p)$ for all primes $p$. The projective image of $r_\pi$ is a finite subgroup of $\PGL_2(\C)$, and is therefore dihedral or isomorphic to a copy of $A_4, S_4$ or $A_5$ (moreover, each of these subgroups is unique up to conjugacy).  We can then establish the automorphy of $\Sym^{n}r_\pi$ case by case, depending on the projective image. In the dihedral case, $r_\pi$ is induced from a character $\psi$ of $G_K$ for $K/\Q$ quadratic, $\Sym^{n}r_\pi$ decomposes as a direct sum of characters and the inductions of characters from $K$ to $\Q$, and therefore $\Sym^{n}r_\pi$ is automorphic. 

In the other cases, we denote the inverse image of $\operatorname{Proj} (r_\pi)(G_{\Q})$ in $\SL_2(\C)$ by $\Gamma^1$. It is a binary polyhedral group. The image $r_{\pi}(G_{\Q})$ is a subgroup of $\mu_{2k}\Gamma^1 \subset \GL_2(\C)$ for some $k$ ($\mu_{2k}$ is the cyclic subgroup of the scalar matrices with order $2k$). We have $\mu_{2k}\Gamma^1 \cong \mu_{2k}\times\Gamma^1/\langle (-1,-I) \rangle$, so its irreducible representations are of the form $\psi\times \sigma$, with $\psi$ a character of $\mu_{2k}$, $\sigma$ an irreducible representation of $\Gamma^1$, and $\psi(-1) = \sigma(-I)$. Twisting by a Dirichlet character, we can assume that $r_\pi(G_\Q) = \mu_{2k}\Gamma^1$ (choose a prime $p$ where $\pi$ is unramified and which is $1$ mod $2k$, then twist by a Dirichlet character with conductor $p$ and order $2k$). Now to understand the decomposition of $\Sym^{n}r_{\pi}$ into irreducibles, it suffices to understand the decomposition of the representation $\Sym^{n}\C^2$ of $\Gamma^1$. See, for example, \cite[Appendix A]{MR2388772} for the character tables of the binary polyhedral groups, or use \cite{GAP4}.

For the $A_5$ case, the irreducible representations of $\Gamma^1$ and their relationship to (symmetric powers of) the two Galois-conjugate irreducible two-dimensional representations are described in \cite[\S5]{Kim2004}.  This allows automorphy of $\Sym^{n}r_\pi$ to be deduced from the automorphy of $\Sym^m$ for $m \le 4$, together with tensor product functorialities $\GL_2 \times \GL_2 \to \GL_4$ and $\GL_2\times\GL_3 \to \GL_6$ \cite[Theorem 6.4]{Kim2004}. 

Now we turn to the $A_4$ case. Considering the character table of the binary tetrahedral group, we see that the irreducible representations of dimension $>1$ comprise: three two-dimensional representations, isomorphic up to twist and a three-dimensional representation which is isomorphic to the symmetric square of the two-dimensional representations. Automorphy of $\Sym^{n}r_\pi$ therefore follows from automorphy of $\Sym^2r_\pi$. 

Finally, in the $S_4$ case, we consider the character table of the binary octahedral group. The irreducible representations of dimension $> 1$ are:

\begin{itemize} \item two faithful two-dimensional representations $V_1, V_2$, isomorphic up to twist,
	\item a two-dimensional representation induced from a character of the normal index two subgroup,
	\item two three-dimensional representations isomorphic to $\Sym^2V_1$ and its twist,
	\item a four-dimensional representation isomorphic to $\Sym^3V_1$. 
\end{itemize}	
So in this case automorphy of general symmetric powers follows from the automorphy of $\Sym^mr_\pi$ for $m \le 3$. 

If $\pi$ is an automorphic induction from a quadratic field $K$, as in the dihedral case, one can construct $\Sym^{n}\pi$ as an isobaric direct sum of Hecke characters and automorphic inductions of Hecke characters for $K$.
\end{proof}

\bibliographystyle{amsalpha}
\bibliography{sym_II}

\renewcommand{\MR}[1]{}
\providecommand{\bysame}{\leavevmode\hbox to3em{\hrulefill}\thinspace}
\providecommand{\MR}{\relax\ifhmode\unskip\space\fi MR }
\providecommand{\MRhref}[2]{%
  \href{http://www.ams.org/mathscinet-getitem?mr=#1}{#2}
}
\providecommand{\href}[2]{#2}
\begin{thebibliography}{BLGGT14}

\bibitem[ANT20]{All19}
Patrick~B. Allen, James Newton, and Jack~A. Thorne, \emph{Automorphy lifting
  for residually reducible {$l$}-adic {G}alois representations, {II}}, Compos.
  Math. \textbf{156} (2020), no.~11, 2399--2422. \MR{4190048}

\bibitem[BBM82]{Ber82}
Pierre Berthelot, Lawrence Breen, and William Messing, \emph{Th\'{e}orie de
  {D}ieudonn\'{e} cristalline. {II}}, Lecture Notes in Mathematics, vol. 930,
  Springer-Verlag, Berlin, 1982. \MR{667344}

\bibitem[BC09]{bellaiche_chenevier_pseudobook}
Jo{\"e}l Bella{\"{\i}}che and Ga{\"e}tan Chenevier, \emph{Families of {G}alois
  representations and {S}elmer groups}, Ast\'erisque (2009), no.~324, xii+314.
  \MR{2656025 (2011m:11105)}

\bibitem[BCDT01]{MR1839918}
Christophe Breuil, Brian Conrad, Fred Diamond, and Richard Taylor, \emph{On the
  modularity of elliptic curves over {$\bold Q$}: wild 3-adic exercises}, J.
  Amer. Math. Soc. \textbf{14} (2001), no.~4, 843--939 (electronic).
  \MR{1839918 (2002d:11058)}

\bibitem[BDJ10]{bdj}
Kevin Buzzard, Fred Diamond, and Frazer Jarvis, \emph{On {S}erre's conjecture
  for mod {$l$} {G}alois representations over totally real fields}, Duke Math.
  J. \textbf{155} (2010), no.~1, 105--161.

\bibitem[BHKT19]{Boc19}
Gebhard B\"{o}ckle, Michael Harris, Chandrashekhar Khare, and Jack~A. Thorne,
  \emph{{$\hat G$}-local systems on smooth projective curves are potentially
  automorphic}, Acta Math. \textbf{223} (2019), no.~1, 1--111. \MR{4018263}

\bibitem[BLGG11]{blgg}
Tom Barnet-Lamb, Toby Gee, and David Geraghty, \emph{The {S}ato-{T}ate
  {C}onjecture for {H}ilbert {M}odular {F}orms}, J. Amer. Math. Soc.
  \textbf{24} (2011), no.~2, 411--469.

\bibitem[BLGGT14]{BLGGT}
Thomas Barnet-Lamb, Toby Gee, David Geraghty, and Richard Taylor,
  \emph{Potential automorphy and change of weight}, Ann. of Math. (2)
  \textbf{179} (2014), no.~2, 501--609. \MR{3152941}

\bibitem[BLGHT11]{cy2}
Tom Barnet-Lamb, David Geraghty, Michael Harris, and Richard Taylor, \emph{A
  family of {C}alabi-{Y}au varieties and potential automorphy {II}}, Publ. Res.
  Inst. Math. Sci. \textbf{47} (2011), no.~1, 29--98. \MR{2827723}

\bibitem[BM02]{breuil-mezard}
Christophe Breuil and Ariane M{\'e}zard, \emph{Multiplicit\'es modulaires et
  repr\'esentations de {${\rm GL}\sb 2({\bf Z}\sb p)$} et de {${\rm
  Gal}(\overline{\bf Q}\sb p/{\bf Q}\sb p)$} en {$l=p$}}, Duke Math. J.
  \textbf{115} (2002), no.~2, 205--310, With an appendix by Guy Henniart.
  \MR{MR1944572 (2004i:11052)}

\bibitem[Car14]{Caraianilp}
Ana Caraiani, \emph{Monodromy and local-global compatibility for {$l=p$}},
  Algebra Number Theory \textbf{8} (2014), no.~7, 1597--1646. \MR{3272276}

\bibitem[Che14]{chenevier_det}
Ga\"etan Chenevier, \emph{The {$p$}-adic analytic space of pseudocharacters of
  a profinite group and pseudorepresentations over arbitrary rings},
  Automorphic forms and {G}alois representations. {V}ol. 1, London Math. Soc.
  Lecture Note Ser., vol. 414, Cambridge Univ. Press, Cambridge, 2014,
  pp.~221--285. \MR{3444227}

\bibitem[CPS75]{Cli75}
Edward Cline, Brian Parshall, and Leonard Scott, \emph{Cohomology of finite
  groups of {L}ie type. {I}}, Inst. Hautes \'{E}tudes Sci. Publ. Math. (1975),
  no.~45, 169--191. \MR{399283}

\bibitem[CT14]{Clo14}
Laurent Clozel and Jack~A. Thorne, \emph{Level-raising and symmetric power
  functoriality, {I}}, Compos. Math. \textbf{150} (2014), no.~5, 729--748.
  \MR{3209793}

\bibitem[DDT97]{DDT}
Henri Darmon, Fred Diamond, and Richard Taylor, \emph{Fermat's last theorem},
  Elliptic curves, modular forms \& {F}ermat's last theorem ({H}ong {K}ong,
  1993), Int. Press, Cambridge, MA, 1997, pp.~2--140. \MR{1605752 (99d:11067b)}

\bibitem[DMW09]{Dum09}
Neil Dummigan, Phil Martin, and Mark Watkins, \emph{Euler factors and local
  root numbers for symmetric powers of elliptic curves}, Pure Appl. Math. Q.
  \textbf{5} (2009), no.~4, Special Issue: In honor of John Tate. Part 1,
  1311--1341. \MR{2560318}

\bibitem[DS74]{deligne-serre}
Pierre Deligne and Jean-Pierre Serre, \emph{Formes modulaires de poids {$1$}},
  Ann. Sci. \'Ecole Norm. Sup. (4) \textbf{7} (1974), 507--530 (1975).
  \MR{MR0379379 (52 \#284)}

\bibitem[GAP20]{GAP4}
The GAP~Group, \emph{{GAP -- Groups, Algorithms, and Programming, Version
  4.11.0}}, 2020.

\bibitem[Gee11]{gee061}
Toby Gee, \emph{Automorphic lifts of prescribed types}, Math. Ann. \textbf{350}
  (2011), no.~1, 107--144. \MR{2785764 (2012c:11118)}

\bibitem[Gel75]{Gel75}
Stephen~S. Gelbart, \emph{Automorphic forms on ad\`ele groups}, Princeton
  University Press, Princeton, N.J.; University of Tokyo Press, Tokyo, 1975,
  Annals of Mathematics Studies, No. 83. \MR{0379375}

\bibitem[Ger19]{ger}
David Geraghty, \emph{Modularity lifting theorems for ordinary {G}alois
  representations}, Math. Ann. \textbf{373} (2019), no.~3-4, 1341--1427.
  \MR{3953131}

\bibitem[GK14]{geekisin}
Toby Gee and Mark Kisin, \emph{The {B}reuil--{M}\'ezard conjecture for
  potentially {B}arsotti--{T}ate representations}, Forum Math. Pi \textbf{2}
  (2014), e1 (56 pages). \MR{3292675}

\bibitem[Kha06]{Kha06}
Chandrashekhar Khare, \emph{Serre's modularity conjecture: the level one case},
  Duke Math. J. \textbf{134} (2006), no.~3, 557--589. \MR{2254626}

\bibitem[Kim04]{Kim2004}
Henry~H. Kim, \emph{An example of non-normal quintic automorphic induction and
  modularity of symmetric powers of cusp forms of icosahedral type}, Invent.
  Math. \textbf{156} (2004), no.~3, 495--502. \MR{2061327}

\bibitem[Kis09a]{MR2551765}
Mark Kisin, \emph{Modularity of 2-adic {B}arsotti-{T}ate representations},
  Invent. Math. \textbf{178} (2009), no.~3, 587--634. \MR{2551765
  (2010k:11089)}

\bibitem[Kis09b]{Kis09}
\bysame, \emph{Moduli of finite flat group schemes, and modularity}, Ann. of
  Math. (2) \textbf{170} (2009), no.~3, 1085--1180. \MR{2600871}

\bibitem[KW09a]{Kha09}
Chandrashekhar Khare and Jean-Pierre Wintenberger, \emph{Serre's modularity
  conjecture. {I}}, Invent. Math. \textbf{178} (2009), no.~3, 485--504.
  \MR{2551763}

\bibitem[KW09b]{Kha09a}
\bysame, \emph{Serre's modularity conjecture. {II}}, Invent. Math. \textbf{178}
  (2009), no.~3, 505--586. \MR{2551764}

\bibitem[Lab11]{labesse}
J.-P. Labesse, \emph{Changement de base {CM} et s\'{e}ries discr\`etes}, On the
  stabilization of the trace formula, Stab. Trace Formula Shimura Var. Arith.
  Appl., vol.~1, Int. Press, Somerville, MA, 2011, pp.~429--470. \MR{2856380}

\bibitem[NT16]{new-tho}
James Newton and Jack~A. Thorne, \emph{Torsion {G}alois representations over
  {CM} fields and {H}ecke algebras in the derived category}, Forum Math. Sigma
  \textbf{4} (2016), e21, 88. \MR{3528275}

\bibitem[NT19]{New19b}
James Newton and Jack~A. Thorne, \emph{Symmetric power functoriality for
  holomorphic modular forms}, 2019, arXiv:1912.11261.

\bibitem[NT20]{New19a}
\bysame, \emph{Adjoint {S}elmer groups of automorphic {G}alois representations
  of unitary type}, 2020, arXiv:1912.11265, to appear in J.~Eur.~Math.~Soc.

\bibitem[{Sta}13]{stacks-project}
The {Stacks Project Authors}, \emph{\itshape {S}tacks {P}roject},
  \texttt{http://stacks.math.columbia.edu}, 2013.

\bibitem[Ste08]{MR2388772}
R.~Stekolshchik, \emph{Notes on {C}oxeter transformations and the {M}c{K}ay
  correspondence}, Springer Monographs in Mathematics, Springer-Verlag, Berlin,
  2008. \MR{2388772}

\bibitem[Tho16]{Tho16}
Jack~A. Thorne, \emph{Automorphy of some residually dihedral {G}alois
  representations}, Math. Ann. \textbf{364} (2016), no.~1-2, 589--648.
  \MR{3451399}

\end{thebibliography}

\end{document}